\documentclass[a4paper,twoside,11pt,reqno]{amsart}
\newcommand{\Ueberschrift}{Categories of abelian varieties over finite fields I \\[1ex] Abelian varieties over $\bF_p$}
\newcommand{\Kurztitel}{Abelian varieties over $\bF_p$}
\usepackage{amsmath} \usepackage{amssymb} \usepackage{amsopn}
\usepackage{amsthm} \usepackage{amsfonts} \usepackage{mathrsfs}
\usepackage{latexsym}

\usepackage[headings]{fullpage}
\usepackage{mathtools}
\usepackage[all]{xy}

\usepackage[dvips]{graphicx} \usepackage[usenames]{color}
\usepackage[OT2, T1]{fontenc}
\usepackage[utf8]{inputenc}
\usepackage[pdfpagelabels, pdftex]{hyperref}
\hypersetup{
  pdftitle={},
  pdfauthor={},
  pdfsubject={},
  pdfkeywords={},
  colorlinks=true,    
  linkcolor=red,     
  citecolor=blue,     
  filecolor=blue,      
  urlcolor=blue,       
  breaklinks=true,
  bookmarksopen=true,
  bookmarksnumbered=true,
  pdfpagemode=UseOutlines,
  plainpages=false}
\DeclareMathOperator{\rH}{H}

\DeclareMathOperator{\rM}{M}

\DeclareMathOperator{\rT}{T}

\newcommand{\bC}{{\mathbb C}}

\newcommand{\bF}{{\mathbb F}}

\newcommand{\bH}{{\mathbb H}}

\newcommand{\bQ}{{\mathbb Q}}
\newcommand{\bR}{{\mathbb R}}

\newcommand{\bZ}{{\mathbb Z}}

\newcommand{\cC}{{\mathscr C}}

\newcommand{\cK}{{\mathscr K}}

\newcommand{\cX}{{\mathscr X}}

\newcommand{\dA}{{\mathcal A}}
\newcommand{\dB}{{\mathcal B}}

\newcommand{\dD}{{\mathcal D}}

\newcommand{\dK}{{\mathcal K}}

\newcommand{\dM}{{\mathcal M}}

\newcommand{\dO}{{\mathcal O}}

\newcommand{\dR}{{\mathcal R}}

\newcommand{\fp}{{\mathfrak p}}

\DeclareSymbolFont{cyrletters}{OT2}{wncyr}{m}{n}
\DeclareMathSymbol{\Sha}{\mathalpha}{cyrletters}{"58}

\newcommand{\surj}{\twoheadrightarrow} 
\newcommand{\inj}{\hookrightarrow}
\DeclareMathOperator{\id}{id}
\DeclareMathOperator{\pr}{pr}
\newcommand{\ev}{{\rm ev}}

\DeclareMathOperator{\Hom}{Hom}

\DeclareMathOperator{\End}{End}

\DeclareMathOperator{\coker}{coker}
\DeclareMathOperator{\im}{im}

\DeclareMathOperator{\Refl}{{\sf Refl}}
\DeclareMathOperator{\AV}{{\sf AV}}

\DeclareMathOperator{\Rk}{rk}

\DeclareMathOperator{\Spec}{Spec}

\DeclareMathOperator{\Pic}{Pic}

\DeclareMathOperator{\ord}{ord}

\DeclareMathOperator{\Ext}{Ext}

\newcommand{\ph}{\varphi}

\newcommand{\com}{\text{\rm com}}

\DeclarePairedDelimiter\abs{\lvert}{\rvert}

\newcommand{\ov}[1]{\mbox{${\overline{#1}}$}} 

\newtheorem{thm}{Theorem}

\newtheorem{prop}[thm]{Proposition}
\newtheorem{lem}[thm]{Lemma}

\newtheorem{cor}[thm]{Corollary}

\theoremstyle{definition}
\newtheorem{defi}[thm]{Definition}

\theoremstyle{remark}
\newtheorem{rmk}[thm]{Remark}

\newtheorem{ex}[thm]{Example}

\newenvironment{pro*}[1][Proof]{{\it{#1:}} }{}
\newenvironment{pro**}[1][]{{\it{#1}} }{\hfill $\square$}

\numberwithin{equation}{section}

\usepackage{enum}

\newlist{enumer}{enumerate}{2}
\setlist[enumer]{label=(\roman*),align=left,labelindent=0pt,leftmargin=*,widest = (iii)}
\newlist{enumerar}{enumerate}{1}
\setlist[enumerar]{label=\arabic*.,align=left,labelindent=0pt,leftmargin=*,widest = 8.}
\newlist{enumera}{enumerate}{2}
\setlist[enumera]{label=(\arabic*),align=left,labelindent=0pt,leftmargin=*,widest = (8)}
\newlist{enumeral}{enumerate}{2}
\setlist[enumeral]{label=(\alph*),align=left,labelindent=0pt,leftmargin=*,widest = (m)}

\begin{document}

\hrule width\hsize
\vskip 0.4cm

\title[\Kurztitel]{\Ueberschrift} 
\author{Tommaso Giorgio Centeleghe}
\address{Tommaso Giorgio Centeleghe, IWR, Universit\"at Heidelberg, Im Neuenheimer Feld 368, 69120 Heidelberg, Germany}
\email{tommaso.centeleghe@iwr.uni-heidelberg.de}

\author{\sc Jakob Stix}
\address{Jakob Stix, Institut f\"ur Mathematik, Johann Wolfgang Goethe--Universit\"at Frankfurt, Robert-Mayer-Stra\ss e~6--8,
60325~Frankfurt am Main, Germany}
\email{stix@math.uni-frankfurt.de}

\thanks{The first author was supported by a project jointly funded by the DFG Priority Program SPP 1489 and the Luxembourg FNR}
\date{\today} 

\maketitle

\begin{quotation} 
\noindent \small {\bf Abstract} --- We assign functorially a $\bZ$-lattice with semisimple Frobenius action to each abelian variety over 
$\bF_p$. This establishes an equivalence of categories that describes abelian varieties over $\bF_p$ avoiding $\sqrt{p}$ as an eigenvalue of Frobenius in terms of
simple commutative algebra. The result extends the isomorphism classification of Waterhouse and Deligne's equivalence for ordinary abelian varieties.
\end{quotation}

\DeclareRobustCommand{\SkipTocEntry}[5]{}

\setcounter{tocdepth}{1} {\scriptsize \tableofcontents}

\section{Introduction}

\subsection{}
\label{sec:set-up}
Let $p$ be a prime number, $\overline\bF_p$ an algebraic closure of the prime field $\bF_p$ with $p$ elements, and $\bF_q\subset\overline\bF_p$
the subfield with $q$ elements, where $q=p^e$ is a power of $p$. The category 
\[
\AV_q
\]
of abelian varieties over $\bF_q$ is an additive category where for all objects $A$, $B$ the abelian groups $\Hom_{\bF_q}(A,B)$ are free of finite rank.
Even though the main result of this paper concerns abelian varieties over the prime field $\bF_p$, the general theme of our work is describing
suitable subcategories $\sf{C}$ of $\AV_q$ by means of lattices $T(A)$ functorially attached to abelian varieties $A$ of $\sf{C}$.
In contrast to the characteristic zero case, if we insist that
\begin{equation}\label{ZrankofT}
\Rk_{\bZ}(T(A))=2\dim(A)
\end{equation}
then it is {\it not} possible to construct $T(A)$ on the whole category $\AV_q$ (see Section~\S\ref{sec:instanceofserreobservation}).
However, if we take $\sf{C}$ to be the full subcategory 
\[
\AV_q^{\rm ord}
\]
of ordinary abelian varieties, Deligne has shown that a functor $A\mapsto T(A)$ satisfying \eqref{ZrankofT} exists and gives an equivalence between $\AV_q^{\rm ord}$
and the category of finite free $\bZ$-modules $T$ equipped with a linear map $F:T\to T$ satisfying a list
of axioms easy to state (cf.~\cite{De}~\S7).

Inspired by a result of Waterhouse (cf.~\cite{Wa}, Theorem~6.1), in the present work we show that a
description in the style of Deligne can in fact be obtained, when $q=p$, for a considerably larger subcategory $\sf{C}$ of $\AV_p$, which excludes only a
single isogeny class of simple objects of $\AV_p$ from occuring as an isogeny factor (see Theorem \ref{MainThm}). Deligne's method is an elegant application of Serre--Tate theory of canonical
liftings of ordinary abelian varieties, whereas our method, closer to that used by Waterhouse, does not involve lifting abelian varieties to characteristic zero.
Even if the main result of this paper generalizes the $q=p$ case of Deligne's theorem, it is unlikely that a proof generalizing Deligne's lifting strategy be possible.

\subsection{}
\label{sec:category}
A \textit{Weil $q$--number} $\pi$ is an algebraic integer, lying in some unspecified field of characteristic zero, such that for any
embedding $\iota : \bQ(\pi) \hookrightarrow\bC$ we have
\[
\abs{\iota(\pi)}=q^{1/2},
\]
where $\abs{-}$ is the ordinary absolute value of $\bC$. Two Weil $q$--numbers $\pi$ and $\pi'$ are
\textit{conjugate} to each other  if there exists an isomorphism $\bQ(\pi) \xrightarrow{\sim} \bQ(\pi')$ 
carrying $\pi$ to $\pi'$, in which case we write $\pi\sim\pi'$. We will denote by 
\[
W_q
\]
the set of conjugacy classes of Weil $q$--numbers. A Weil $q$--number is either totally real or totally imaginary, hence it makes sense to speak of a \textit{non-real} element of $W_q$.

Let $A$ be an object of $\AV_q$, denote by $\pi_A:A\to A$ the Frobenius isogeny of $A$ relative to $\bF_q$. 
If $A$ is $\bF_q$--simple then 
$\End_{\bF_q}(A)\otimes\bQ$ is a division ring,
and a well known result of Weil says that $\pi_A$ is a Weil $q$--number inside 
the number field $\bQ(\pi_A)$. 
Let
\begin{equation}\label{Poincarereducibility}
A\sim\prod_{1\leq i\leq r} A_i^{e_i}
\end{equation}
be the decomposition of $A$ up to $\bF_q$--isogeny into powers of simple, pairwise non--isogenous factors $A_i$.
The \emph{Weil support} of $A$ is defined as the subset
\[
w(A)=\{\pi_{A_1}, \ldots, \pi_{A_r}\} \subseteq W_q
\]
given by the conjugacy classes of the Weil numbers $\pi_{A_i}$ attached to the simple factors $A_i$.
By Honda--Tate theory, the conjugacy classes of the $\pi_{A_i}$ are pairwise distinct, moreover any class in $W_q$ arises as
$\pi_{A}$, for some $\bF_q$--simple abelian variety $A$, uniquely determined up to $\bF_q$--isogeny (cf. \cite{Ta}, Th\'eor\`eme 1).

\subsection{}
\label{sec:equivalence}
Consider now the case $q=p$. Using Honda--Tate theory it is easy to see that for a simple object $A$ of $\AV_p$ the ring
$\End_{\bF_p}(A)$ is commutative if and only if $\pi_A\not\sim\sqrt p$, i.e., 
if and only if the Frobenius isogeny $\pi_A:A\to A$ defines a non-real Weil $p$--number (cf.~\cite{Wa}, Theorem~6.1). Let
\[
\AV_p^{\rm com}
\]
be the full subcategory of $\AV_p$ given by all objects $A$ such that $w(A)$ does not contain the conjugacy class of $\sqrt p$.
Equivalently,
$\AV_p^{\rm com}$ is the largest full subcategory of $\AV_p$ closed under taking cokernels containing all simple objects
whose endomorphism ring is commutative. Since the Weil $p$--number $\sqrt p$ is associated to an $\bF_p$--isogeny class of simple, supersingular abelian surfaces
(cf. \cite{Ta}, exemple (b) p.~97), we have a natural inclusion $\AV_p^{\rm ord}\subset\AV_p^{\rm com}$.

\smallskip

The main result of this paper, proven at the end of Section~\S\ref{sec:proofmain}, is the following.

\begin{thm} 
\label{MainThm} 
There is an ind-representable contravariant functor  
\[
A \mapsto  (T(A),F)
\]
which induces an anti-equivalence between $\AV_p^{\rm com}$ and the category of
pairs $(T, F)$ given by a finite, free $\bZ$-module $T$ and an endomorphism $F:T\to T$ satisfying the following properties.
\begin{enumer}
\item $F\otimes\bQ$ is semisimple, and its eigenvalues are non-real Weil $p$--numbers.
\item There exists a linear map $V:T\to T$ such that $FV=p$.
\end{enumer}
Moreover, the lattice  $T(A)$ has rank $2\dim(A)$ for all $A$ in $\AV_p^{\rm com}$, and $F$ is equal to $T(\pi_A)$.
\end{thm}

In order to prove the theorem we consider in Section~\S\ref{GorensteinBis} a family of Gorenstein rings 
\[
R_w = \bZ[F,V]/(FV-p, h_w(F,V))
\]
indexed by the finite subsets $w \subseteq W_p$, where $h_w(F,V)$ is a certain symmetric polynomial built out of the minimal polynomials over $\bQ$ of the elements of $w$.
An object $(T,F)$ in the target category of the functor $T(-)$ of Theorem \ref{MainThm} is nothing but an $R_w$-module, for $w\subset W_p$ large enough, that is free of finite rank as a
$\bZ$-module. In this translation, the linear map $F:T\to T$ is given by the action of the image of $F$ in $R_w$, and the relation $h_w(F,V)$ in $R_w$ encodes precisely that $F \otimes \bQ$ acts
semisimply and with eigenvalues given by Weil $p$-numbers lying in $w$ (see Sections~\S\ref{sec:symmetricpoly}, ~\S\ref{sec:structureminimalcentral} and \S\ref{sec:targetcategory}).
Thanks to the Gorenstein property, these $R_w$-modules are precisely the reflexive $R_w$-modules, the category that they form will be denoted by (see Section~\S\ref{sec:reflexive})
\[
\Refl(R_w).
\]

For  $v \subseteq w$, the corresponding rings are linked by natural surjective maps
$\pr_{v,w}:R_w\longrightarrow R_v$. 
We denote by $\dR_p^\com$ the pro-system $(R_w,\pr_{v,w})$ with $w \subseteq W_p$ ranging over the finite subsets avoiding the conjugacy class of $\sqrt{p}$. We further set 
\[
\Refl(\dR_p^\com) = \varinjlim_{w \subseteq W_p\setminus\{\sqrt p\}} \Refl(R_w).
\]
We refer to Section~\S\ref{sec:targetcategory} for details.

In this language, Theorem~\ref{MainThm} can be stated as saying that
\[
T  : \AV_p^\com \to \Refl(\dR_p^\com)
\]
is an anti--equivalence of categories. While this formulation of the main result is closer to the perspective we adopted in its proof,
the more concrete statement we chose to give above allows an immediate 
comparison to Deligne's result in \cite{De}.

\subsection{}
\label{sec:theringsRw}

The rings $R_w$ studied in Section~\S\ref{GorensteinBis} are in fact defined for any finite subset $w\subseteq W_q$.  They appear
naturally in connection to abelian varieties, in that for any $A$ in $\AV_q$ the natural map
\begin{equation} \label{eq:R-structure}
\bZ[F,V]/(FV-q)\longrightarrow\End_{\bF_q}(A)
\end{equation}
sending $F$ to $\pi_A$ and $V$ to the Verschiebung isogeny $q/\pi_A$ induces an identification between $R_{w(A)}$ and the subring $\bZ[\pi_A,q/\pi_A]$
of $\End_{\bF_q}(A)$, which has finite index in the center (see Section 2.1).
The rings $R_w$ have been already considered in \cite{Wa} and
\cite{howe:ppav} for example, however, to our best knowledge, our observation that they are 
almost\footnote{When $[\bF_q:\bF_p]=e$ is even we must require that the set $w$ either contains both or none of the two rational Weil $q$--numbers $\pm q^{e/2}$. In particular, $R_w$ is Gorenstein for a cofinal subsystem of the finite subsets $w \subseteq W_q$ 
 with respect to inclusion.} 
always Gorenstein remained so far unnoticed (see Theorem~\ref{thm:gorenstein}).

An $\dR_p^\com$-linear structure on $\AV_p^\com$ can be deduced from the map \eqref{eq:R-structure} (see Section~\S\ref{linearstructures}).
The requirement that $F=T(\pi_A)$ precisely means
that the functor $T(-)$ is an $\dR_p^\com$-linear functor (see Section ~\S\ref{sec:targetcategory}).

\subsection{}
\label{sec:outline of proof} 
The proof of the theorem consists of two steps. First, for any finite subset $w\subseteq W_p$
not containing the conjugacy class of $\sqrt p$, we construct a certain abelian variety $A_w$ isogenous to the product of all simple objects
attached to the elements of $w$ via Honda--Tate theory. 
The object $A_w$ is chosen in its isogeny class with the smallest possible endomorphism ring, i.e., such that the natural map 
\[
R_w \to \End_{\bF_p}(A_w)
\]
is an isomorphism (see Proposition \ref{ConstructionApi}). In order to show the existence of such an $A_w$, which already appears in \cite{Wa} Theorem~6.1
if $w$ consists of a single element, the assumption $q=p$ plays an important role.
Exploiting the Gorenstein property of $R_w$, in Theorem~\ref{thm:T(A)onAVpi} we show that the functor
$\Hom_{\bF_p}(-,A_w)$ gives a contravariant equivalence
\[
\Hom_{\bF_p}(-,A_w):\AV_w\stackrel{\sim}{\longrightarrow}\Refl(R_w)
\]
where $\AV_w$ is the full subcategory of $\AV_p$ given by all abelian varieties $A$  with $w(A)\subseteq w$.

The second step consists in showing that the abelian varieties $A_w$ previously constructed can be chosen in such a way that the functors $\Hom_{\bF_p}(-,A_w)$
interpolate well, and define a functor on $\AV_p^\com$. More precisely we show the existence of an ind-system
\begin{equation}\label{inductivesystem}
\dA= (A_w, \ph_{w,v}),
\end{equation}
indexed by finite subsets $w\subseteq W_p$ not containing the conjugacy class of $\sqrt p$, 
such that the corresponding direct limit of finite free $\bZ$--modules
\[
T(A)=\varinjlim_w\Hom_{\bF_p}(A, A_w)
\]
stabilizes. The contravariant functor $T(-)$ ind-represented by $\dA$  will produce the required anti--equivalence.

\subsection{}
\label{sec:instanceofserreobservation}

As Serre has observed, it is {\it not} possible to functorially construct a lattice $T(A)$ satisfying the expected 
$\Rk_{\bZ}(T(A))=2\dim(A)$
on the category of abelian varieties over $\overline\bF_p$.
This is due to the existence of objects like supersingular elliptic curves $E$ over $\overline\bF_p$. As is well known, the division ring
$\End_{\overline\bF_p}(E) \otimes \bQ$ is  a non-split quaternion algebra over $\bQ$ and has no $2$-dimensional $\bQ$--linear representation that can serve as $T(E) \otimes \bQ$.
That just described is the same obstruction that prevents the existence of a Weil--cohomology for varieties over finite fields with rational coefficients.

Using the same argument one can show the non-existence of a lattice $T(A)$ as above on the category $\AV_q$, where $q$ is a square.
When $q$ is not a square, the correct instance of Serre's observation, which prevents Theorem~\ref{MainThm} from extending to all of $\AV_p$, is
given by the isogeny class of $\bF_q$-simple, supersingular abelian surfaces associated via Honda--Tate theory to the real, non rational,
Weil $q$--number $\sqrt q$. The endomorphism ring of any such surface $A$ is an order of a quaternion algebra over $\bQ(\sqrt{q})=\bQ(\sqrt{p})$ which is ramified at the two real
places (cf.~\cite{Wa} p.~528). It follows that $\End_{\bF_q}(A)\otimes\bR \simeq \bH \times \bH$ is a product of two copies of the Hamilton quaternions $\bH$. Thus it admits no
faithful representation on a $4$-dimensional real vector space, such as $T(A)\otimes\bR$ would give rise to. 

\subsection{}
The dual abelian variety establishes an anti--equivalence $A\mapsto A^t$ of $\AV_q$ which preserves Weil supports and has the effect
of switching the roles of Frobenius and Verschiebung endomorphisms relative to $\bF_q$. This is to say that
\[
(\pi_A)^t=q/\pi_{A^t}
\]
as isogenies from $A^t$ to itself. On the module side, we define a covariant involution of $\Refl(\dR_p^\com)$ denoted by $M\mapsto M^\tau$ which interchanges the roles of $F$ and $V$, i.e., such that 
\[
{(T,F)}^\tau=(T,p/F).
\]

Using these two dualities we can exhibit a covariant version of the functor $T(-)$ of Theorem~\ref{MainThm}. More precisely, define
\[
T_\ast(A)= T(A^t)^\tau  
\]
as the pair given by the $\bZ$--module $T(A^t)$ equipped with the linear map $p/T(\pi_{A^t})$.
In the notation as pairs $T_\ast(A)$ takes the form
\[
(T(A^t), p/T(\pi_{A^t})) = (T(A^t), T((\pi_A)^t)) = (T_\ast(A),T_\ast(\pi_A)). 
\]
The functor $T_\ast(-)$ gives a covariant, $\dR_p^\com$--linear equivalence
\begin{equation}\label{covariantversion}
T_\ast : \AV_p^\com \to \Refl(\dR_p^\com)
\end{equation}
which is pro-represented by the system $\dA^t = (A_w^t, \ph_{w',w}^t)$ dual to \eqref{inductivesystem}.
In the definition of $T_\ast(-)$ it is necessary to apply the involution $\tau$ to $T(A^t)$ in order to guarantee that $T_\ast$ be $\dR_p^\com$--linear.

In Section~\S\ref{sec:comparedeligne} we compare $T_\ast(-)$ restricted to $\AV_p^{\ord}$ with Deligne's functor from \cite{De} \S7 that we denote by $T_{{\rm Del},p}(-)$. The comparison makes use of a compatible pro-system of projective $R_w$-modules $M_w$ of rank $1$ for all finite subsets $w \subseteq W_p$ consisting only of conjugacy classes of ordinary Weil $p$--numbers. Proposition~\ref{prop:compareDeligne} then describes, for all abelian varieties $A$ over $\bF_p$ with $w(A) \subseteq w$, a natural isomorphism 
\[
T_{{\rm Del,p}}(A)  \otimes_{R_w}  M_w   \xrightarrow{\sim} T_\ast(A).
\]
Furthermore, by choosing a suitable ind-representing system $\dA = (A_w,\ph_{v,w})$ we may assume that $M_w = R_w$ for all $w$, i.e., the anti-equivalence of Theorem~\ref{MainThm} may be chosen to extend in its covariant version Deligne's equivalence, see Proposition~\ref{prop:extendDeligne} for details.

\subsection{}
Finally, we indicate how to recover the $\ell$-adic Tate module 
$T_\ell(A)$,
for a prime $\ell\neq p$, and the contravariant Dieudonn\'e module 
$T_p(A)$
(cf. \cite{Wa}~\S1.2) from the module $T(A)$. This involves working with the formal Tate module $T_\ell(\dA)$ and the formal Dieudonn\'e module $T_p(\dA)$
of the direct system $\mathcal{A}$, respectively defined as the direct limit of $T_\ell(A_w)$ and the inverse limit of the $T_p(A_w)$, with transition maps obtained via
functoriality of $T_\ell$ and $T_p$. More concretely we have natural isomorphisms
\begin{align*}
T_\ell(A) & \simeq 
\Hom_{\mathcal{R}_\ell}(T(A)\otimes\bZ_\ell, T_\ell(\mathcal{A})), \\
T_p(A) & \simeq  (T(A) \otimes \bZ_p) \hat{\otimes}_{\dR_p} T_p(\dA),
\end{align*}
see Proposition~\ref{prop:recovertate} and \ref{prop:recoverdieudonne} for notation and proofs.
In this respect the functor $T(-)$  can be interpreted as an integral lifting of the Dieudonn\'e module functor
$T_p(-)$.

In a forthcoming paper we will apply the method used here to study certain categories of abelian varieties over a finite field which is larger that $\bF_p$. 
Therefore, although Theorem~\ref{MainThm} deals with abelian varieties over $\bF_p$, we only restrict to the case $q=p$ when it becomes necessary.

\smallskip

\addtocontents{toc}{\SkipTocEntry}
\subsection*{Acknowledgments} 
The authors would like to thank Gebhard B\"ockle for stimulating discussions and for his suggestion of the symmetric polynomial $h_\pi(F,V)$. We thank Filippo
Nuccio for valuable comments on an earlier version of the manuscript, and Hendrik Lenstra for his interesting observations on local complete intersections. 
Special thanks go to Brian Conrad and Frans Oort for their attentive reading of a preliminary version of our work, and for the prompt and interesting feed back they gave us. Finally
we thank the anonymous referees for their quick work in reviewing the paper.

\section{On the ubiquity of Gorenstein rings among minimal central orders}
\label{GorensteinBis}

\subsection{Minimal central orders}
\label{sec:defineRw}
Let $w\subseteq W_q$ be any finite set of conjugacy classes of Weil $q$--numbers. Choose Weil $q$--numbers $\pi_1,\ldots,\pi_r$
representing the elements of $w$, and consider the ring homomorphism
\begin{equation}\label{RtoProd}
\bZ[F,V]/(FV-q)\longrightarrow\prod_{1\leq i\leq r}\bQ(\pi_i)
\end{equation}
sending $F$ to $(\pi_1,\ldots,\pi_r)$ and $V$ to $(q/\pi_1,\ldots,q/\pi_r)$.

\begin{defi} 
\label{defi:Rw}
The \emph{minimal central order $R_w$} is the quotient 
\begin{equation} \label{eq:defineRw}
\bZ[F,V]/(FV-q) \to R_w
\end{equation}
by the kernel of the homomorphism \eqref{RtoProd}. The image of $F$ in $R_w$ will be denoted by $F_w$, and the image of $V$ by $V_w$.
\end{defi}
The construction of the ring $R_w$ is independent of the chosen Weil $q$--numbers in their respective conjugacy classes. 
When $w$ consists of a single conjugacy class of a Weil number $\pi$, the ring $R_{\{\pi\}}$, isomorphic to the order of $\bQ(\pi)$ generated by $\pi$ and $q/\pi$,
will sometimes be denoted simply by $R_\pi$. Since the representatives $\pi_1,\ldots,\pi_r$ are pairwise non--conjugate, there is a canonical finite index inclusion
\[
R_w \subseteq \prod_{\pi \in w}R_{\pi}, 
\]
in particular 
\begin{equation} \label{eq:rationalRw}
R_w \otimes \bQ = \prod_{\pi \in w} \bQ(\pi).
\end{equation}
Moreover, for finite subsets $v \subseteq w\subseteq W_q$ we have a natural surjection
 \[
\pr_{v,w}:R_{w}\longrightarrow R_v.
 \]
 
Our main goal in this section is showing that, under a mild assumption on $w$, the ring $R_w$ is a one dimensional Gorenstein ring.
This will be proved in Section \ref{sec:structureminimalcentral}, where we obtain a description of $R_w$ by identifying the relations between the
generators $F$ and $V$.

\begin{ex}\label{ex:congruences} The equality of closed subschemes
\[
\Spec(R_w) = \bigcup_{\pi \in w}  \Spec(R_\pi) \subseteq \Spec \big(\bZ[F,V]/(FV-q)\big)
\]
shows that the spectrum of $R_w$ is obtained by glueing the spectra of the rings $R_\pi$ along their various intersections inside
$\Spec \big(\bZ[F,V]/(FV-q)\big)$. This roughly means that it is the congruences between Weil $q$-numbers who are responsible for
$R_w$ differing from the product of the $R_\pi$ for all $\pi\in w$.

We measure in a special situation the deviation of $R_w$ from being
isomorphic to $\prod_{\pi \in w} R_\pi$. Let $\pi_i$ for $ i =1,2$ be quadratic Weil $q$-numbers with minimal polynomial 
\[
x^2 - \beta_i x + q,
\]
where $\beta_i\in\bZ$, and set $\Delta = \beta_1 - \beta_2$. Since $q/\pi_i = \beta_i - \pi_i$ we have 
\[
R_{\pi_i} = \bZ[\pi_i] \simeq \bZ[x]/(x^2 - \beta_i x + q),
\]
moreover the subring $R_w\subseteq \bZ[\pi_1] \times \bZ[\pi_2]$ is generated as a $\bZ$-algebra by 
\[
(0,\Delta), (\pi_1,\pi_2) \in \bZ[\pi_1] \times \bZ[\pi_2],
\]
since it is generated by $(\pi_1,\pi_2)$ and $(\beta_1-\pi_1,\beta_2-\pi_2)$.
Because $\beta_1 \equiv \beta_2$ modulo $\Delta$, there are isomorphisms of  quotients 
\[
\bZ[\pi_1]/ \Delta \bZ[\pi_1] \simeq \bZ[\pi_2]/ \Delta \bZ[\pi_2] =: R_0
\]
and $R_w$ becomes the fibre product 
\[
R_w = \bZ[\pi_1] \times_{R_0} \bZ[\pi_2],
\]
which is an order  of index $\Delta^2$ in the product $R_{\pi_1} \times R_{\pi_2}$. The congruences between $\pi_1$ and $\pi_2$ 
are encoded by the closed subscheme of $\Spec \big(\bZ[F,V]/(FV-q)\big)$ given by
\[
\Spec(R_0) = \Spec(R_{\pi_1}) \cap \Spec(R_{\pi_2}).
\]
Note that the minimal polynomials $x^2 - \beta_i + q$ yield Weil $q$-numbers if and only if 
\[
\beta_i^2 < 4q.
\]
In particular, by letting $q$ range over the powers of the prime $p$, the Weil $q$-numbers $\pi_i$ may be chosen so as to have
$\Delta$ divisible by an arbitrary integer.
\smallskip
\end{ex}

\subsection{Connection to abelian varieties} We proceed to link $R_w$ to abelian varieties over $\bF_q$. Any such $A$ has two distinguished isogenies
given by the Frobenius $\pi_A$ and the Verschiebung $q/\pi_A$ relative to $\bF_q$.
The $\bQ$--algebra $\End_{\bF_q}(A)\otimes\bQ$ is semi--simple, and its center is equal to the sub--algebra $\bQ(\pi_A)$ generated by $\pi_A$ (cf. \cite{Ta2}, Theorem~2).
It follows that any isogeny decomposition of $A$, as in \eqref{Poincarereducibility}, induces the isomorphism
\begin{equation}\label{DecompositionCenter}
\bQ(\pi_A)\simeq \prod_{\pi_{A_i} \in w(A)} \bQ(\pi_{A_i}),
\end{equation}
sending $\pi_A$ to $(\pi_{A_1},\ldots,\pi_{A_r})$, where $\pi_{A_1},\ldots,\pi_{A_r}$ are the Weil $q$--numbers defined by the simple factors of $A$, and $w(A)$ is the Weil support of $A$ as defined in the introduction.

From \eqref{DecompositionCenter} we deduce that the ring homomorphism
 \[
r_A:\bZ[F,V]/(FV-q)\longrightarrow\End_{\bF_q}(A)
 \]
sending $F$ to $\pi_A$ and $V$ to $q/\pi_A$ gives an identification between $R_{w(A)}$ and the image of $r_A$, namely the subring
 \[
\bZ[\pi_A,q/\pi_A]
 \]
which sits inside the center of $\End_{\bF_q}(A)$ with finite index. In this way we see that $R_{w(A)}$ plays the role of a lower bound 
for the center of $\End_{\bF_q}(A)$. This justifies the terminology we chose in its definition.

\begin{rmk}\label{occouranceofRw} One can raise the question of whether there exists an abelian variety $A$ with Weil support
$w$ such that the natural map $R_{w}\to\End_{\bF_p}(A)$ induced by $r_A$ gives an isomorphism between $R_w$ and the
center of $\End_{\bF_p}(A)$. In Proposition \ref{ConstructionApi} below, generalizing a result of Waterhouse, we obtain a partial
result in this direction.
\end{rmk}

\subsection{Linear structures over minimal central orders}\label{linearstructures} For a finite subset $w\subseteq W_q$ the full subcategory
\[
\AV_w\subseteq\AV_q
\]
consists of all abelian varieties $A$ such that $w(A)\subseteq w$ or, equivalently, such that
$r_A$ factors through the quotient $\bZ[F,V]/(FV-q)\to R_w$. Since for any morphism $f:A\to B$ in $\AV_q$ and
any $\eta\in\bZ[F,V]/(FV-q)$ the diagram
 \begin{equation}\label{diagramnaturalityfrob}
 \xymatrix@M+1ex{
 A \ar[r]^{f}\ar[d]_{r_A(\eta)} & B\ar[d]^{r_B(\eta)}\\
 A \ar[r]^{f}& B\\
 }
 \end{equation}
is commutative, as follows from the naturality of the Frobenius and Verschiebung isogenies, we deduce an $R_w$--linear structure on the category $\AV_w$.
Furthermore, for finite subsets $v \subseteq w$ the $R_w$--linear structure on $\AV_v$ induced by the fully faithful inclusion
$\AV_v \subseteq \AV_{w}$ is compatible, via the surjection $\pr_{v,w}$, with the $R_v$--linear structure on $\AV_v$.

\begin{rmk}\label{rmk:RwstructureviaFrobenius}
If $W\subseteq W_q$ is now any subset, denote by $\dR_W$ the projective system $(R_w,\pr_{w,v})$ as $w$ ranges through all finite subsets of $W$, and by 
\[
\AV_W
\]
the full subcategory of $\AV_q$ whose objects are all abelian varieties $A$ with $w(A)\subseteq W$. We will treat $\AV_W$ as the direct $2$-limit of the categories $\AV_w$,
for $w$ a finite subsets of $W$. The collection of $R_w$--linear structures on the subcategories $\AV_w\subseteq\AV_W$, which are linked by the
compatibility conditions described above, form what we will refer to as an $\dR_W$--linear structure on $\AV_W$.
\end{rmk}

\subsection{The symmetric polynomial} 
\label{sec:symmetricpoly}

Let $\pi$ be a Weil $q$--number. If $\bQ(\pi)$ has a real place then $\pi^2 = q$, so that $\bQ(\pi)$ is totally real, and $[\bQ(\pi):\bQ]$ is either $2$ or $1$
according to whether the degree $e=[\bF_q:\bF_p]$ is odd or even, respectively. In the first case there is only one conjugacy
class of real Weil $q$--numbers, in the second one there are two of them, given by the rational integers $q^{e/2}$ and $-q^{e/2}$.
In the general case where $\pi$ is not real, the field $\bQ(\pi)$ is a non-real CM field, with complex conjugation
induced by $\pi\mapsto q/\pi$.

The degree $2d = [\bQ(\pi):\bQ]$ is even, except for the two rational Weil $q$--numbers occurring for $e$ even. Set 
\[
P_\pi(x) = x^{2d} + a_{2d-1}x^{2d-1} + \ldots + a_1 x + a_0 \in \bZ[x] 
\]
for the normalized minimal polynomial of $\pi$ over $\bQ$, and accept that $d = 1/2$ in case $\pi\in\bZ$. The polynomial $P_\pi(x)$
depends only on the conjugacy class of $\pi$. The following lemma is well known (cf. \cite{howe:ppav}, Prop. 3.4).

\begin{lem} \label{lem:integers}
Let $\pi$ be a non-real Weil $q$--number. For $r\geq 0$, we have $a_{d-r} = q^r a_{d+r}$.
\end{lem}
\begin{proof} We can arrange the roots $\alpha_1, \ldots, \alpha_{2d}$ of $P_\pi(x)$ so that $\alpha_i$ and $\alpha_{2d+1-i}$ are complex conjugates of each
other, which is to say $\alpha_i \alpha_{2d+1-i} = q$. For a subset $I \subseteq \{1,\ldots, 2d\}$ we set $I^c = \{1,\ldots,2d\} \setminus I$, and
$\ov{I} = \{i \ ; \ 2d+1 - i \in I\}$, and moreover we use the multiindex notation $\alpha^I = \prod_{i \in I} \alpha_i$. Then with summation over subsets of
$\{1,\ldots,2d\}$ we compute
\begin{align*}
(-1)^{d+r} a_{d-r} = \sum_{|I| = d+r} \alpha^I  & = \big(\prod_{i=1}^{2d} \alpha_i \big) \cdot  \sum_{|I| = d+r}  \frac{1}{\alpha^{I^c}} \\
& = q^{r} \cdot  \sum_{|J| = d - r}  \frac{q^{d-r}}{\alpha^{J}} = q^r \cdot \sum_{|J| = d - r}  \alpha^{\ov{J}} = q^r (-1)^{d-r} a_{d+r},
\end{align*}
and this proves the lemma.
\end{proof}

We next construct a symmetric polynomial $h_\pi(F, V)\in\bZ[F, V]$. 
The idea is to consider the rational function $P_\pi(F)/F^d \in \bZ[F, q/F]$ (at least when $d\in\bZ$), and then formally set $V=q/F$.

\begin{defi}
We define the \textbf{symmetric polynomial} $h_\pi(F,V)$ attached to a Weil $q$--number $\pi$ as follows:
\begin{enumerate} 
\item
If $\pi$ is a non-real Weil $q$--number, then we set  in $\bZ[F, V]$
\[
h_\pi(F,V) = F^{d} + a_{2d-1}F^{d-1} + \ldots + a_{d+1}F+a_d + a_{d+1}V + \ldots + a_{2d-1}V^{d-1} +  V^{d}.
\]
\item
If $\pi = \pm p^m \sqrt{p}$ is real but not rational, then we set 
\[
h_\pi(F,V) = F - V \in \bZ[F,V].
\]
 \item
If $\pi  = \pm p^m$ is rational,  then we set
\[
h_{\pm p^m}(F,V) =  F^{1/2} \mp V^{1/2} \in \bZ[F^{1/2}, V^{1/2}].
\]
\end{enumerate}
\end{defi}

The polynomial $h_w(F,V)$ just defined appears already in \cite{howe:ppav}, \S 9.

\begin{lem} \label{lem:vanishinghpi} 
\begin{enumerate}
\item 
If $\pi$ is a non-real Weil $q$--number, then we have
$h_\pi(\pi, q/\pi)=0$.
\item 
If $\pi$ is a real, but not rational Weil $q$-number, then $h_\pi(F,V) = F - V$ and $h_\pi(\pi,q/\pi) = 0$.
\item
If $\pi  = \pm p^m$ is rational, then 
$h_{p^m}(F,V) \cdot h_{-p^m}(F,V) = F - V $
is again contained in $\bZ[F,V]$, and vanishes for $F= \pi$ and $V = q/\pi$.
\end{enumerate}
\end{lem}
\begin{proof}
Assertion (1) follows from $h_\pi(\pi,q/\pi) = P_\pi(\pi)/\pi^d = 0$ which is based on Lemma~\ref{lem:integers}. Assertion (2) and (3) are trivial.
\end{proof}

\begin{defi}
An \textbf{ordinary} Weil $q$--number is a Weil $q$--number $\pi$ such that exactly half of the roots of its minimal polynomial $P_\pi(x)$ in an algebraic
closure of $\bQ_p$ are $p$-adic  units. Equivalently, the isogeny class of abelian varieties over $\bF_q$ associated to $\pi$ by Honda--Tate theory is ordinary.
\end{defi}

If $\pi$ is ordinary then $\bQ(\pi)$ is not real, and precisely half of the roots of the even degree polynomial $P_\pi(x)$ are $p$-adic units.

\begin{lem}
\label{lem:criterionordinary}
Let $w \subseteq W_q$ be a finite subset of non-real conjugacy classes of Weil $q$--numbers. 
Then $w$ consists of ordinary conjugacy classes, if and only if $h_w(0,0)$ is not divisible by $p$.
\end{lem}
\begin{proof}
Let $\alpha_1,\ldots,\alpha_d, q/\alpha_1,\ldots, q/\alpha_d$ be the roots of $\prod_{\pi \in w} P_\pi(x)$. Then 
\[
h_w(F,V) \equiv  \prod_{i = 1}^d (F  - (\alpha_i + q/\alpha_i) + V)  \mod (FV - q)
\]
so that 
\[
h_w(F,V) \equiv (-1)^d  \prod_{i = 1}^d (\alpha_i + q/\alpha_i)  \mod p.
\]
This integer is not divisible by $p$ if and only if for all $i$ the algebraic integers $\alpha_i + q/\alpha_i$ are $p$-adic units. This happens if and only if for all $i$ either $\alpha_i$ or $q/\alpha_i$ are $p$-adic units, hence if $w$ consists of ordinary conjugacy classes.
\end{proof}

\subsection{Structure of the minimal central orders}
\label{sec:structureminimalcentral}
In what follows we will define for a finite subset $w \subseteq W_q$ the degree of $w$ by 
\[
\deg(w) = \Rk_\bZ(R_w) = \sum_{\pi \in w} [\bQ(\pi):\bQ].
\]
So $w$ is of even degree if and only if $w$ either contains none or both rational Weil $q$--numbers $\pm q^{e/2}$,
which only show up when $e=[\bF_q:\bF_p]$ is even. Extending this notion, we will say that an arbitrary subset $W \subseteq W_q$ is of even degree if either none or both rational conjugacy classes of Weil $q$-numbers belong to $W$.

If $w\subseteq W_q$ is any finite subset we set 
\[
h_w(F, V) = \prod_{\pi \in w} h_{\pi}(F,V),
\]
which is contained in $\bZ[F,V]$ as soon as $w$ is of even degree.

\begin{thm} 
\label{thm:gorenstein}
Let $w \subseteq W_q$  be a finite set of Weil $q$--numbers of even degree.
\begin{enumera}
\item 
\label{thmitem:structure}
We have $R_w = \bZ[F,V]/(FV - q, h_w(F, V))$.
\item 
\label{thmitem:Gorensteinconditiongeneralpiandq}
The ring $R_w$ is a $1$-dimensional complete intersection, in particular it is a Gorenstein ring.
\end{enumera}
\end{thm}

\begin{proof}
The ring $R_w$ is reduced as it injects into a product of number fields. Moreover, $R_w$ is a finite $\bZ$-algebra, because it is generated by $F$ and $V$ that satisfy integral relations in $R_w$. Thus $R_w$ is free of finite rank as a $\bZ$-module and of Krull dimension $1$. More precisely, by \eqref{eq:rationalRw} we have 
\[
\Rk_\bZ(R_w) = \sum_{\pi \in w} [\bQ(\pi):\bQ] =: 2D
\]

The ring $\bZ[F,V]/(FV - q)$ is a normal ring with at most one rational singularity in $\fp = (F,V,p)$. Hence, $h_w(F, V)$ is a non-zero divisor in $\bZ[F,V]/(FV-q)$ and it remains to 
show~\ref{thmitem:structure} to conclude the proof of~\ref{thmitem:Gorensteinconditiongeneralpiandq}.

We now show assertion~\ref{thmitem:structure}. By Lemma~\ref{lem:vanishinghpi} the evaluation of  $h_w(F,V)$  in $R_\pi$ vanishes for all $\pi \in w$. Hence we obtain a surjection
\[
\ph \ : \  S= \bZ[F,V]/(FV - q,h_w(F, V)) \surj R_w.
\]
We are done if we can show that $S$ is generated by $2D$ elements as a $\bZ$-module. 

By construction, $h_w(F,V)$ is a product of polynomials of the form 
\[
f_\pi(F) + g_\pi(V)
\]
with $f_\pi,g_\pi \in \bZ[X]$ monic (or $-g_\pi$ monic). The degrees are  $\deg(f_\pi)  = \deg(g_\pi) = [\bQ(\pi):\bQ]/2$ if $\pi$ is non-rational, and $1$ if $\pi$ is rational.
Having a representative of the form $f(F)+g(V)$ for monic polynomials $f,g$ (or $-g$) of the same degree is preserved under taking products:
\[
\big(f_1(F) +g_1(V)\big) \big(f_2(F) +g_2(V)\big) = f_1f_2(F) + g_1g_2(V) + \text{lower degree terms in $F,V$,}
\]
where the mixed terms are of lower degree, because $FV = q$ necessarily leads to cancellations.

Hence the same holds for the product: $h_w(F,V) = f(V) + g(V)$ with $\deg(f) = \deg(g) = D$. In particular 
\[
F^D, F^{D-1}, \ldots, F, 1, V, \ldots, V^{D-1}
\]
generate $S$ as a $\bZ$-module.
\end{proof}

Part \ref{thmitem:structure} of Theorem~\ref{thm:gorenstein} has already been observed by Howe, at least for $w$ ordinary (cf. \cite{howe:ppav}, Prop. 9.1).
On the other hand the Gorensteinness of $R_w$, so crucial in the present work, seems to have remained unnoticed so far.

Since Theorem~\ref{MainThm} deals with abelian varieties over $\bF_p$, our main concern in this paper are the commutative algebra properties of $R_w$
for finite subsets of $W_p$. Here Theorem~\ref{thm:gorenstein} covers all cases. In order to complete the picture we answer what happens if $w \subseteq W_q$
contains exactly one rational conjugacy class of Weil $q$--numbers.

\begin{thm}
Let $q$ be the square of a positive or negative  integer $\sqrt{q} \in \bZ$. Let $v \subseteq W_q$ be a finite set containing no rational conjugacy class, and set
$w = v \cup \{\sqrt{q}\}$. Then the following holds.
\begin{enumera}
\item 
\label{thmitem:structuregeneral}
We have $R_w = \bZ[F,V]/(FV - q, h_v(F,V)(F -\sqrt{q}), h_v(F,V)(V-\sqrt{q}))$.
\item 
\label{thmitem:Gorensteinconditiongeneralpiandqgeneral}
The ring $R_w$ is Gorenstein if and only if all conjugacy classes of Weil $q$-numbers in $v$ are ordinary.
\end{enumera}
\end{thm}
\begin{proof}
Reasoning as in Lemma~\ref{lem:vanishinghpi}, we see that the defining quotient map $\bZ[F,V]/(FV-q) \to R_w$ factors as a surjective map 
\[
S = \bZ[F,V]/(FV - q, h_v(F,V)(F -\sqrt{q}), h_v(F,V)(V-\sqrt{q})) \surj R_w.
\]
As in Theorem~\ref{thm:gorenstein} the ring $R_w$ as a $\bZ$-module is free of rank 
\[
\Rk_\bZ(R_w) = 1 + \sum_{\pi \in v} [\bQ(\pi):\bQ] =: 2D + 1.
\]
It is easy to see that $S$ is generated as a $\bZ$-module by 
\[
F^D, F^{D-1}, \ldots, F, 1, V, \ldots, V^{D}.
\]
This shows assertion \ref{thmitem:structuregeneral} as above.

For assertion \ref{thmitem:Gorensteinconditiongeneralpiandqgeneral} we first note that after inverting one of the elements $p$, $F$ or $V$ the three relations can be reduced to two relations, so that outside of $(p,F,V)$ the ring $R_w$ is a local complete intersection and hence Gorenstein.  It remains to discuss the local ring in $\fp = (p,F,V)$. 

There is a unique polynomial $h \in \bZ[X]$ such that
\[
h_v(F,V) = h(F) - h(0) + h(V) \in \bZ[F,V],
\]
and for this $h$ we have $h(0) = h_v(0,0)$. Since $\bZ$ is regular (hence Gorenstein) and $R_w$ is a flat $\bZ$-algebra, it follows from \cite{matsumura:ringtheory}  Theorem~23.4 that $R_w$ is Gorenstein in $\fp$ if and only if 
\[
R_w/p R_w = \bF_p[F,V]/(FV, h(F)F, h(V)V)
\]
is Gorenstein in $\bar{\fp} = (F,V)$. The ring $R_w/pR_w$ is Artinian, hence of dimension $0$, so that by \cite{matsumura:ringtheory}  Theorem~18.1 the ring  $(R_w/pR_w)_{\bar \fp}$ is Gorenstein if and only if
\[
1 =  \dim_{\bF_p} \Hom(\kappa(\bar \fp),R_w/pR_w).
\]
The space of homomorphisms has the same dimension as the socle,  i.e., the maximal submodule annihilated by $(F,V)$.
The socle is the intersection of the kernels of $F$ and $V$ as $\bF_p$-linear maps of $R_w$, that can be easily evaluated in the basis $F^D, F^{D-1}, \ldots, F, 1, V, \ldots, V^{D}$. The  intersection is $1$-dimensional if $p \nmid h(0)$ and it is $2$-dimensional otherwise. Due to Lemma~\ref{lem:criterionordinary} this completes the proof.
\end{proof}

\section{Remarks on reflexive modules}
\label{sec:reflexive}

\subsection{Reflexive versus $\bZ$-free} 

Let $S$ be a noetherian ring. Recall that a finitely generated $S$-module $M$ is \textit{reflexive} (resp.\ \textit{torsion less}) if  the natural map 
\[
M \to \Hom_S(\Hom_S(M,S),S) 
\]
is an isomorphism (resp.\ injective). We denote the category of finitely generated reflexive $S$-modules by 
\[
\Refl(S).
\]

\begin{lem} 
\label{lem:reflexiveZfree}
Let $w \subseteq W_q$ be a finite set of Weil $q$--numbers such that $R_w$ is Gorenstein, 
and let $\ell$ be a prime number. 
Let $M$ be a finitely generated $R_w$-module (resp.\ $R_w \otimes \bZ_\ell$-module). The following are equivalent:
\begin{enumeral}
\item \label{lemitem:reflexive}
$M$ is reflexive.
\item \label{lemitem:torsionless}
$M$ is torsion less.
\item \label{lemitem:Zfree}
$M$ is free as a $\bZ$-module (resp.\ $\bZ_\ell$-module).
\end{enumeral}
\end{lem}
\begin{proof}
Assertions~\ref{lemitem:reflexive} and \ref{lemitem:torsionless} are equivalent by \cite{Ba} Theorem 6.2 (4), since $R_{w}$ is Gorenstein and of dimension $1$.

For a uniform treatment, we set $S = R_w \otimes \Lambda$ with $\Lambda = \bZ$ (resp.\ $\Lambda = \bZ_\ell$). Since $S$ is finite flat over $\Lambda$, the dual module
$\Hom_S(M,S)$ is free as a $\Lambda$-module. The same holds for every submodule of $\Hom_S(M,S)$ which shows
assertion~\ref{lemitem:torsionless} implies \ref{lemitem:Zfree}. 

For the converse direction we introduce the total ring of fractions $S \subset K = S \otimes_\bZ \bQ$, which is a product of fields. 
Therefore, assuming~\ref{lemitem:Zfree}, the composite map 
\[
M \to M \otimes_\bZ \bQ = M \otimes_S K \to \Hom_S(\Hom_K(M \otimes_S K,K),K)
\]
is injective.  And since it factors over the natural map $M \to \Hom_S(\Hom_S(M,S),S)$, the latter is also injective and hence $M$ torsion less. This completes the proof.
\end{proof}

\subsection{The main theorem with reflexive modules}
\label{sec:targetcategory}

Let $w \subseteq W_q$ be a finite set of conjugacy classes of Weil $q$--numbers of even degree (see~\S\ref{sec:structureminimalcentral}), so that, in particular, $R_w$ is Gorenstein (see Theorem~\ref{thm:gorenstein}). For
an object $M$ of $\Refl(R_w)$, let $(M_0,F_M)$ be the pair consisting of the $\bZ$--module $M_0$ underlying $M$ and of the linear map $F_M:M_0\to M_0$ given by the action of $F_w\in R_w$ on $M$.

\begin{prop}\label{RwModTF}
The functor $M\mapsto (M_0, F_M)$ gives
an equivalence between $\Refl(R_w)$ and the category of pairs $(T, F)$ consisting of a finite, free $\bZ$-module $T$, and of an endomorphism
$F:T\to T$ satisfying the following conditions:
\begin{enumer}
\item $F\otimes\bQ$ is semi-simple with eigenvalues given by Weil $q$--numbers in $w$;
\item there exists $V:T\to T$ such that $FV=q$.
\end{enumer}
A morphism between two such pairs $(T,F)$ and $(T', F')$ is a linear map $f:T\to T'$ such that
$f F=F'f$.
\end{prop}
\begin{proof} Thanks to Lemma~\ref{lem:reflexiveZfree}, an $R_w$--module belongs to $\Refl(R_w)$ if and only if it is finite and free as a $\bZ$--module.
Moreover, the linear map $F_M:M_0\to M_0$ satisfies in the ring $\End_{\bZ}(M_0)$ the polynomial
\[
F^d \cdot h_w(F,q/F) = \prod_{\pi \in w} P_\pi(F),
\]
which is square free. Therefore $F_M\otimes\bQ$ is semi-simple with eigenvalues given by Weil $q$--numbers whose conjugacy classes
belong to $w$. The map $V_M:M_0\to M_0$ induced by the action of $V_w\in R_w$ on $M$ satisfies $V_MF_M=q$. Essential surjectivity
of the functor follows easily from Lemma~\ref{lem:reflexiveZfree}.
\end{proof}

Let now $v \subseteq w$ be a finite subset which is also of even degree. By Lemma~\ref{lem:reflexiveZfree}, the natural projection $\pr_{v,w} : R_{w} \to R_{v}$ gives a fully faithful embedding 
\[
\Refl(R_v) \subseteq \Refl(R_{w})
\]
by means of which $\Refl(R_v)$ can be regarded as the full subcategory whose objects are those for which the $R_{w}$--action factors over
$\pr_{v,w} : R_{w} \to R_{v}$. Using the description of Proposition~\ref{RwModTF}, we easily see that an object $M$ of $\Refl(R_w)$ lies in $\Refl(R_v)$ if and only if
the eigenvalues of $F_M\otimes\bQ:M_0\otimes\bQ\to M_0\otimes\bQ$ define conjugacy classes of Weil $q$--numbers in $v$.
\smallskip

\begin{defi} Let $W \subseteq W_q$ be a subset of even degree, and $\dR_W = (R_w)$ be the pro--ring with $w$ ranging over all finite subsets of $W$ of even degree. The category
\[
\Refl(\dR_W) := \varinjlim_{w \subseteq W} \Refl(R_{w})
\]
is the full subcategory of the category of $\bZ[F,V]$--modules given by all $M$ such that:

\begin{enumerate}

\item there exists $w_M\subseteq W$ such that the structural action of $\bZ[F,V]$ on $M$ factors through
$\bZ[F,V]\to R_{w_M}$ (and hence through $\bZ[F,V]\to R_w$ for all $w$ containing $w_M$);

\item\label{condition2} for any finite  $w\subseteq W$ of even degree containing $w_M$ the module $M$ is reflexive as an $R_w$--module.
\end{enumerate}
\end{defi}
Notice that condition \eqref{condition2} is equivalent to asking that $M$ be a reflexive module over $R_w$ for some $w\subseteq W$  of even degree such that the action of $R_w$ on $M$ is defined (see Lemma \ref{lem:reflexiveZfree}).

\begin{rmk}\label{complin} For any finite $w\subseteq W$ of even degree, the category $\Refl(\dR_W)$ contains the $R_w$--linear category $\Refl(R_{w})$ as a full subcategory.
Moreover, if $v\subseteq w$ are finite subsets of $W$ of even degree, then the $R_{v}$--linear structure on $\Refl(R_v)$ induced from the fully faithful embedding
$\Refl(R_v)\subseteq \Refl(R_{w})$ is compatible, via the surjection $\pr_{v,w}:R_{w}\to R_v$, with the natural $R_w$--linear structure. Formally we are
in a situation analogous to that described in Remark \ref{rmk:RwstructureviaFrobenius}, where the category $\AV_W$ played the role of $\Refl(\dR_W)$.
We will then say that $\Refl(\dR_W)$ has a $\dR_W$--linear structure.
\end{rmk}

The category $\Refl(\dR_W)$ can be given a concrete description in terms of pairs $(T, F)$
given by a finite free $\bZ$-module $T$ and a linear map $F:T\to T$ such that
\begin{enumer}
\item $F\otimes\bQ$ is semi-simple and its eigenvalues are Weil $q$-numbers in $W$;
\item there exists $V:T\to T$ with $FV=q$.
\end{enumer}
The notion of morphism between two such pairs is clear. This can be seen reasoning as in Proposition~\ref{RwModTF},
and using the compatibility of linear structures described in Remark~\ref{complin}.

\smallskip

Denote now the set $W_p\setminus \{\sqrt{p}\}$ of non--real conjugacy classes of Weil $p$-numbers simply by $W_p^\com$, and the corresponding
pro-ring $\dR_{W_p^\com}$ by $\dR_p^\com$. Theorem~\ref{MainThm} then claims the existence of a contravariant, $\dR_{W_p^\com}$-linear,
ind-representable equivalence 
\[
T: \AV_p^\com \to \Refl(\dR_p^\com),
\]
such that $T(A)$ is a lattice of rank $2 \dim(A)$. By definition, the $\dR_{W_p^\com}$-linearity of $T(-)$ is the requirement that for any finite
$w\subseteq W_p^\com$ the restriction of $T$ to $\AV_w$ has values in $\Refl(R_w)$ and is $R_w$--linear. These conditions amounts precisely to the equality
$F  = T(\pi_A)$, for all $A$ in $\AV_p^\com$.

\subsection{Further remarks}

The following piece of homological algebra is used later.

\begin{lem} \label{lem:vanishingext1}
Let $S$ be a $1$-dimensional Gorenstein ring. For any  finitely generated reflexive $S$-module $M$, we have  
\[
\Ext^1_S(M,S) = 0.
\]
\end{lem}
\begin{proof}
We use a free presentation of the dual $\Hom_S(M,S)$ and dualize again. This yields an embedding of $M$ into a free $S$-module and then a short exact sequence
\[
0 \to M \to S^n \to M' \to 0.
\]
The $\Ext$-sequence, and the fact that $S$ has injective dimension $1$, cf.~\cite{Ba} \S1, yields
\[
0 = \Ext^1_S(S^n,S) \to \Ext^1_S(M,S) \to \Ext^2_S(M',S) = 0
\]
from which the lemma follows.
\end{proof}

Finally, here is a criterion for invertible reflexive modules in terms of their endomorphism algebra.

\begin{prop} \label{prop:EndCriterionForInvertible}
Let $S$ be a reduced Gorenstein ring of dimension at most $1$, and let $M$ be a reflexive module. Then the following are equivalent.
\begin{enumerate}
\item[(a)] $M$ is locally free of rank $1$.
\item[(b)] The natural map $S \to \End_S(M)$ is an isomorphism.
\end{enumerate}
\end{prop}
\begin{proof}
If $M$ is locally free of rank $1$, then $\End_S(M) \simeq M^\vee \otimes M \simeq S$ where $M^\vee = \Hom_S(M,S)$ and (b) holds. 

For the converse, we may assume that $S$ is a complete local ring by passing to the completion.  Since $\End_S(M) = S$ we have $M \not= 0$, and  moreover, $M$ cannot be a module (extending the $S$-module structure) for a strictly larger subring of the total ring of fractions of $S$. Now \cite{Ba} Proposition~7.2 shows that  $M$ has a non-zero projective direct summand $M_0$. With $M = M_0 \oplus M_1$ we find 
\[
S \times \End_S(M_1) =  \End_S(M_0) \times \End_S(M_1)  \subseteq  \End_S(M) = S
\]
and therefore $\End_S(M_1) = 0$. This forces $M_1 = 0$ and $M$ is projective. Then $\End_S(M)$ is projective of rank the square of the rank of $M$ (as a locally constant function on $\Spec(S)$). Therefore $M$ is of rank $1$ and the proof is complete.
\end{proof}

\section{Abelian varieties with minimal endomorphism algebra}
\label{TheFunctorTonAVpi}

Before restricting to the case $q=p$, we recall the following classical result of Tate (cf. \cite{Ta2} \S1 for $\ell \not= p$, \cite{WM} including $\ell=p$: Theorem~6, also \cite{CCO}~\S A.1)
which will be used frequently. For $A$ an abelian variety over $\bF_q$ and $\ell$ a prime number, denote by $A[\ell^\infty]$ the $\ell$-divisible group corresponding to $A$.

\begin{thm}[Tate]
\label{TateThm} Let $A, B$ be abelian varieties over $\bF_q$, and $\ell$ a prime number. The natural map $f\mapsto f[\ell^\infty]$ induces an isomorphism
\[
\Hom_{\bF_q}(A,B)\otimes\bZ_\ell\stackrel{\sim}{\longrightarrow}\Hom(A[\ell^\infty],B[\ell^\infty]).
\]
\end{thm}

As is well known, the isomorphism of Tate's theorem takes a more concrete form as follows.
If $\ell\neq p$,
it can be formulated in terms of Galois representations, and says that the functor $\ell$-adic Tate-module $T_\ell(-)$ induces an isomorphism
\[
\Hom_{\bF_q}(A,B)\otimes\bZ_\ell\stackrel{\sim}{\longrightarrow}\Hom_{\bZ_\ell[{\rm Gal}_{\bF_q}]}(T_\ell(A),T_\ell(B)).
\]

If $\ell=p$, using the language of Dieudonn\'e modules, Tate's theorem translates into the fact that the functor contravariant Dieudonn\'e-module $T_p(-)$ induces an isomorphism
\[
\Hom_{\bF_q}(A,B)\otimes\bZ_p\stackrel{\sim}{\longrightarrow}\Hom_{\mathcal{D}_{\bF_q}}(T_p(B),T_p(A)),
\]
where $\mathcal{D}_{\bF_q}$ is the Dieudonn\'e ring of $\bF_q$.

\begin{rmk}\label{localRwStructure}
For any prime $\ell$ the $R_w$-linear structure on the category $\AV_w$ defined in Section~\S\ref{sec:defineRw} induces an enrichment of the functor
$T_\ell(-)$ to left $R_w\otimes\bZ_\ell$-modules for $\ell\neq p$, and to right\footnote{We employ the contravariant Dieudonn\'e theory, therefore the left $R_w$-module
structure of the $\Hom$-groups in $\AV_w$ turns into a right $R_w\otimes\bZ_p$-modules structure on the corresponding Dieudonn\'e modules. However $R_w$ is commutative,
hence for $A$ in $\AV_w$ we can safely treat $T_p(A)$ as a left $R_w\otimes\bZ_p$-module.}
$R_w\otimes\bZ_p$-modules for $\ell=p$.

For any $A\in\AV_w$, and any
$\ell\neq p$, the action of the arithmetic Frobenius of $\bF_q$ on $T_\ell(A)$ agrees with the action of $F_w\otimes 1\in R_w\otimes\bZ_\ell$, and we have
a natural identification
\[
\Hom_{\bZ_\ell[{\rm Gal}_{\bF_p}]}(T_\ell(A),T_\ell(B))=\Hom_{R_w\otimes\bZ_\ell}(T_\ell(A),T_\ell(B))\text{ for $\ell\neq p$},
\]
for all $A,B\in\AV_w$. In the special case where $q=p$, and only in this case, the Dieudonn\'e ring $\dD_{\bF_q}$ is commutative, hence the theory of
Dieudonn\'e modules of abelian varieties over the prime field $\bF_p$ does not involve semi-linearity aspects. For any $A\in\AV_w$ the action of
$\dD_{\bF_p}$ on $T_p(A)$ factors through the quotient $\dD_{\bF_p} \surj R_w \otimes \bZ_p$, and Tate's theorem reads as
\[
\Hom_{\mathcal{D}_{\bF_p}}(T_p(A), T_p(B))=\Hom_{R_{w}\otimes\bZ_p}(T_p(B), T_p(A)),
\]
for all $A,B\in\AV_w$. So that, roughly, the Dieudonn\'e theory of abelian varieties over the prime field is analogous to the
theory of Tate modules at primes $\ell\neq p$, up to replacing variancy with covariancy.
\end{rmk}

For any $\pi\in W_p$, we choose  a simple abelian variety $B_\pi$ over $\bF_p$ whose associated Weil 
$p$--number represents $\pi$.

\begin{prop}
\label{ConstructionApi} 
Let $w \subseteq W_p$ be a finite set of conjugacy classes of Weil $p$--numbers not containing $\sqrt p$.
There exists an abelian variety $A_w$ over $\bF_p$ isogenous to $\prod_{\pi \in w} B_{\pi}$ and such that $T_\ell(A_{w})$ is free of rank one over $R_{w} \otimes \bZ_\ell$, for all primes $\ell$.
Furthermore, for any such $A_w$, the natural map 
\[
R_w \to  \End_{\bF_p}(A_{w})
\]
is an isomorphism.
\end{prop}

\begin{rmk}
In the case where $w$ consists of just one Weil $p$--number, the abelian variety $A_{w}$ as in Proposition~\ref{ConstructionApi}  was already considered by
Waterhouse (cf. \cite{Wa} Theorem~6.1). We observe that the product $\prod_{\pi\in w} A_{\{\pi\}}$ of the varieties constructed for each singleton $\{\pi\}\subset w$
may well fail to serve as the $A_w$ satisfying the properties of Proposition~\ref{ConstructionApi}. This failure is explained by a phenomenon analogous
to congruences between Weil $q$-numbers, discussed in example~\ref{ex:congruences}.
\end{rmk}

\begin{proof} Let $B$ be any abelian variety over $\bF_p$ isogenous to $\prod_{\pi \in w} B_{\pi}$. For any $\pi\in W_p$ with
$\pi\not\sim\sqrt p$, it is straightforward to verify using Honda--Tate theory (cf. \cite{Ta}, Th\'eor\`eme 1~ii)) that:

\begin{enumer}
\item all local invariants of the division ring $\End_{\bF_p}^0(B_\pi)$ are trivial;

\item $[\bQ(\pi):\bQ]=2\dim(B_\pi)$.
\end{enumer}
In fact each of these conditions is equivalent to $\End_{\bF_p}(B_\pi)$ being commutative.

Since the abelian varieties $B_\pi$,
for $\pi\in w$, are pairwise non--isogenous, we have that $\End_{\bF_p}(B)$ is also commutative, and isomorphic to an order of the product of CM--fields $\prod_{\pi\in w} \bQ(\pi)$.
We deduce the chain of equalities
\[
\Rk_\bZ(\End_{\bF_p}(B)) 
= \sum_{\pi \in w} [\bQ(\pi):\bQ] = \sum_{\pi \in w} 2 \dim(B_\pi) = 2 \dim(B).
\]

From the injectivity of the isomorphism of Theorem~\ref{TateThm}, and using the language of Dieudonn\'e modules if $\ell=p$, it follows that the action of
$R_{w}\otimes\bQ_\ell=\prod_{\pi \in w} \bQ(\pi)\otimes\bQ_\ell$ on 
\[
V_\ell(B) = T_\ell(B) \otimes_{\bZ_\ell} \bQ_\ell
\]
is faithful. Hence $V_\ell(B)$ has rank one over $\prod_{\pi \in w} \bQ(\pi)\otimes\bQ_\ell$,
since they both have dimension $2\dim(B)$ over $\bQ_\ell$ (notice that $\dim_{\bQ_p}(V_p(B))=2\dim(B)$ because we work over $\bF_p$).

Therefore, for every $\ell$,  we can choose an $R_{w}\otimes\bZ_\ell$-lattice 
\[
\Lambda_\ell\subset V_\ell(B)
\]
which is free of rank one, and which contains $T_\ell(B)$ if $\ell\neq p$, and is contained in $T_p(B)$ if $\ell=p$.

If $R_{w}\otimes\bZ_\ell$ is the maximal order of $\prod_{\pi \in w} \bQ(\pi) \otimes \bQ_\ell$, as for almost all $\ell$, then $T_\ell(B)$ is necessarily free of rank one over $R_{w}\otimes\bZ_\ell$ and we take $\Lambda_\ell = T_\ell(B)$.

Now, if $\ell\neq p$ then the finite subgroup 
\[
N_\ell=\Lambda_\ell/T_\ell(B)\subset B[\ell^\infty],
\]
being an  $R_w$-submodule, is stable under Frobenius and hence is defined over $\bF_p$. The corresponding isogeny
$\psi_\ell:B\to B/N_\ell$ induces an identification $\Lambda_\ell\simeq T_\ell(B/N_\ell)$ of $R_{w}\otimes\bZ_\ell$-modules. 

Similarly,
the $p$-power degree isogeny $\psi_p:B\to B/N_p$, where $N_p$ is the $\bF_p$-subgroup-scheme of $B$ corresponding to the Dieudonn\'e module
$T_p(B)/\Lambda_p$, induces an identification $T_p(B/N_p)\simeq\Lambda_p$ of $R_{w}\otimes\bZ_p$-modules. Therefore, after applying a
finite sequence of isogenies to $B$, we obtain the abelian variety $A_{w}$ with the desired property. 

Lastly, by Theorem~\ref{TateThm}, the natural map
\[
R_w \to \End_{\bF_p}(A_w)
\]
is an isomorphism after $-\otimes\bZ_\ell$ for all prime numbers $\ell$, since $T_\ell(A_w) \simeq R_w \otimes \bZ_\ell$. Therefore the last
statement of the proposition follows.
\end{proof}

\begin{rmk} 
One can show that there is a free and transitive action of the Picard group $\Pic(R_{w})$ on the set of
isomorphism classes of abelian varieties $A_{w}$ satisfying the conditions of Proposition~\ref{ConstructionApi} (cf. \cite{Wa} Theorem~6.1.3 for the case of simple abelian varieties, i.e., $w = \{\pi\}$). We will discuss this below in 
Section~\S\ref{sec:choicesofindrepresentingobjects}.
\end{rmk}

The Gorenstein property of $R_{w}$ allows one to give the following useful characterization of the abelian varieties $A_{w}$
satisfying the property of Proposition~\ref{ConstructionApi} (cf. also \cite{serretate:goodreduction} end of \S 4).

\begin{prop}
\label{characterizationApi} 
Let $w \subseteq W_p$ be a finite set of conjugacy classes of Weil $p$--numbers not containing $\sqrt p$, and let $A$ be
an abelian variety over $\bF_p$ isogenous to $\prod_{\pi \in w} B_{\pi}$. The following conditions are equivalent:
\begin{enumeral}
\item\label{characterizationApi1} $T_\ell(A)$ is free of rank one over $R_{w}\otimes\bZ_\ell$, for all primes $\ell$.
\item\label{characterizationApi2} $\End_{\bF_p}(A)$ is equal to the minimal central order $R_{w}$.
\end{enumeral}
\end{prop}
\begin{proof} Thanks to Proposition~\ref{ConstructionApi}, we only need to show that \ref{characterizationApi2} implies \ref{characterizationApi1}.
Since $R_w$ is Gorenstein by Theorem~\ref{thm:gorenstein}, also its completion $R_w \otimes \bZ_\ell$ is Gorenstein. It follows from \cite{Ba} Theorem~6.2 that
the torsion free $R_w \otimes \bZ_\ell$-module $T_\ell(A)$ is reflexive. 

By \ref{characterizationApi2} and Theorem~\ref{TateThm} we have $\End_{R_w \otimes \bZ_\ell}(T_\ell(A)) = R_w \otimes \bZ_\ell$, so Proposition~\ref{prop:EndCriterionForInvertible} yields that $T_\ell(A)$ is projective of rank $1$. Since $R_w \otimes \bZ_\ell$ is a finite $\bZ_\ell$-algebra, hence a product 
\[
R_w \otimes \bZ_\ell = \prod_\lambda R_\lambda
\]
of complete local rings  $R_\lambda$, its Picard group is trivial and $T_\ell(A)$ is free of rank $1$ as an $R_w \otimes \bZ_\ell$-module.
\end{proof}

We conclude the section observing that if $A$ is an abelian variety over $\bF_q$, for $q$ arbitrary, the Dieudonn\'e module $T_p(A)$ has rank
$2\dim(A)$ over the Witt vectors $W(\bF_q)$ of $\bF_q$. It follows that the naive analogue of \ref{characterizationApi1} can never be
attained if $q>p$ for rank reasons, and the above proposition is peculiar to the $q=p$ case.

\section{Construction of the anti-equivalence}
\label{sec:anitequivalence}

In this section we give a proof of Theorem~\ref{MainThm}. 
Recall from Remark~\ref{rmk:RwstructureviaFrobenius} that for a subset $W \subseteq W_q$ of conjugacy classes of Weil $q$-numbers, the category  $\AV_W$
is the full subcategory of $\AV_q$ consisting of all abelian varieties $A$ over $\bF_q$  whose support $w(A)$ is contained in $W$.

\subsection{Finite Weil support} 

We begin by defining the lattice $T(A)$ and its endomorphism $F$ on the increasing family of subcategories 
\[
\AV_w \subseteq \AV_p^{\rm com}
\]
for  finite subsets $w \subseteq W_p^\com$. 

Let us then assume that $\sqrt p \notin w$, and pick an abelian variety $A_w$ satisfying the
condition of Proposition~\ref{ConstructionApi} for $w$. For any object $A$ of $\AV_w$ there is a natural $R_w = \End_{\bF_p}(A_w)$-module structure on  
\[
\rM_w(A) := \Hom_{\bF_p}(A, A_w).
\]
This is the same $R_w$-structure as that described in Remark~\ref{rmk:RwstructureviaFrobenius}.

\begin{thm}\label{thm:T(A)onAVpi} 
Let $w \subseteq W_p$ be a finite set of non-real conjugacy classes of Weil $p$--numbers. 
The functor $\rM_w(A)$
induces an anti-equivalence 
\[
\AV_w \to \Refl(R_w).
\]
The $\bZ$-rank of $\rM_w(A)$ is $2\dim(A)$. 
\end{thm}

\begin{proof} 
We begin by showing that $\rM_w(-)$ is fully faithful. The map 
\[
f: \Hom_{\bF_p}(A', A'') \longrightarrow \Hom_{R_\pi}(\rM_w(A''), \rM_w(A'))
\]
is a homomorphism of finitely generated $\bZ$-modules, hence it is an isomorphism if and only if it is an isomorphism after scalar extension $-\otimes \bZ_\ell$ for all primes $\ell$.

We first treat the case $\ell \not= p$. For $N \in \Refl(R_w \otimes \bZ_\ell)$, set
\[
N^\vee = \Hom_{R_w \otimes \bZ_\ell}(N,T_\ell(A_w))
\]
which is isomorphic to the $R_w\otimes\bZ_\ell$-dual of $N$, in view of our choice of $A_w$. The isomorphism of Theorem~\ref{TateThm} gives a natural isomorphism of contravariant  functors 
\begin{equation} \label{eq:elladictatefromT}
(T_\ell(-))^\vee = \Hom_{R_w \otimes \bZ_\ell}(T_\ell(-),T_\ell(A_w)) \simeq \rM_w(-) \otimes \bZ_\ell
\end{equation}
on $\AV_w$ (see Remark \ref{localRwStructure}). This translates into the following commutative diagram:
\[
\xymatrix@M+1ex{
\Hom_{\bF_p}(A',A'') \otimes \bZ_\ell \ar[r]^\simeq \ar[d]^{f \otimes \bZ_\ell}  & \Hom_{R_w \otimes \bZ_\ell}(T_\ell(A'),T_\ell(A'')) \ar[d]^{(-)^\vee}  \\
\Hom_{R_w}(\rM_w(A''),\rM_w(A')) \otimes \bZ_\ell \ar[r]^\simeq  & \Hom_{R_w \otimes \bZ_\ell}(T_\ell(A'')^\vee,T_\ell(A')^\vee)
}
\]
where both horizontal maps are isomorphisms as a consequence of Theorem~\ref{TateThm}. Since $R_w \otimes\bZ_\ell$ is a completion of a Gorenstein ring by Theorem~\ref{thm:gorenstein}, it is  itself Gorenstein. Because $T_\ell(A_w)$ is free of rank $1$, this implies that $N \mapsto N^\vee$ is an contravariant autoequivalence of $\Refl(R_w \otimes \bZ_\ell)$, 
see \cite{Ba} Theorem~6.2. Therefore the  right vertical map in the diagram is an isomorphism and we conclude that $f \otimes \bZ_\ell$ is an isomorphism as well.

Concerning the case $\ell = p$, for any $N\in\Refl(R_w\otimes\bZ_p)$ we set
\[
N_\vee = \Hom_{R_w \otimes \bZ_p}(T_p(A_w),N).
\]
The isomorphism of Theorem~\ref{TateThm} then gives a natural isomorphism of contravariant functors 
\begin{equation} \label{eq:dieudonnefromT}
(T_p(-))_\vee = \Hom_{R_w \otimes \bZ_p}(T_p(A_w),T_p(-)) \simeq \rM_w(-) \otimes \bZ_p
\end{equation}
on $\AV_w$, which translates into the following commutative diagram:
\[
\xymatrix@M+1ex{
\Hom_{\bF_p}(A',A'') \otimes \bZ_p \ar[r]^\simeq \ar[d]^{f \otimes \bZ_p}  & \Hom_{R_w \otimes \bZ_p}(T_p(A''),T_p(A')) 
\ar[d]^{(-)_\vee}  \\
\Hom_{R_w}(\rM_w(A''),\rM_w(A')) \otimes \bZ_p \ar[r]^\simeq  & \Hom_{R_w \otimes \bZ_p}(T_p(A'')_\vee,T_p(A')_\vee) .
}
\]
The horizontal maps are isomorphisms by Theorem~\ref{TateThm}.  Since $T_p(A_w)$ is free of rank $1$ over $R_w \otimes \bZ_p$,   the  right vertical map in the diagram is an isomorphism. We conclude that $f \otimes \bZ_p$ is an isomorphism as well.

We have now established that the functor $A\mapsto\rM_w(A)$ from $\AV_w$ to the category
$\Refl(R_w)$ is fully faithful.

\smallskip

In order to show that $\rM_w(-)$ is an equivalence we must now show that this functor is essentially surjective. Let $M \in \Refl(R_w)$ be a reflexive
module. Since $R_w$ is Gorenstein, the natural map $M \to \Hom_{R_w}(\Hom_{R_w}(M, R_w), R_w)$ is an isomorphism. Dualizing a presentation of the dual $\Hom_{R_w}(M, R_w)$, leads to a co-presentation
\[
0 \to M \to (R_w)^n \xrightarrow{\psi} (R_w)^m.
\]
Since $\rM_w(A_w) = \End_{\bF_p}(A_w) = R_w$ we find by fully faithfulness of $\rM_w(-)$ a homomorphism 
\[
\Psi : (A_w)^m \to (A_w)^n 
\]
with $\psi = \rM_w(\Psi)$. The cokernel 
\[
B = \coker(\Psi)
\]
exists and is an abelian variety $B \in \AV_w$. By definition of the cokernel, the functor $\rM_w(-)$ is left-exact, hence 
\[
0 \to \rM_w(B) \to \rM_w((A_w)^n) \xrightarrow{\rM_w(\Psi)} \rM_w((A_w)^m)
\]
and so 
\[
M \simeq \rM_w(B)
\]
as $R_w$-modules. This completes the proof of essential surjectivity.

We are only left with showing that $\Rk_\bZ(\Hom_{\bF_p}(A, A_w)) = 2\dim(A)$, for all $A$ in $\AV_w$. 
The statement is additive in $A$ and depends only on the isogeny class of $A$ and $A_w$. 
Recall that for any $\pi\in W_p$, we have chosen 
 a simple abelian variety $B_\pi$ over $\bF_p$ whose associated Weil $p$--number represents $\pi$.
Because $A_w$ is isogenous to 
$\prod_{\pi\in w} B_\pi$, it is enough to show that for
any $\pi\in w$ we have
\[
\Rk_\bZ(\Hom_{\bF_p}(B_{\pi}, \prod_{\pi'\in w} B_{\pi'})) = 2\dim(B_\pi).
\]
This follows from $\Rk_\bZ(\End_{\bF_p}(B_\pi))=2\dim(B_\pi)$ for all Weil $p$--numbers $\pi\not\sim\sqrt p$ (cf. \cite{Ta}, Th\'eor\`eme 1 ii)),
and the proof of the theorem is complete.
\end{proof}

\subsection{The direct system}
\label{sec:proof}

In order to prove Theorem~\ref{MainThm} we construct a direct system $\mathcal{A}=\varinjlim A_{w}$ consisting of abelian varieties
$A_{w}$ indexed by finite sets $w$ of  Weil $p$--numbers not containing $\sqrt p$, and having the property stated in Proposition
\ref{ConstructionApi}.

Let $v \subseteq w$ be two finite sets of non-real Weil $p$--numbers. By means of the canonical surjection
\[
\pr_{v,w} : R_{w} \surj R_{v},
\]
we may consider $R_{v}$-modules as $R_{w}$-modules such that the action factors over $\pr_{v,w}$. 
Lemma~\ref{lem:reflexiveZfree} shows that 
\[
\Refl(R_{v}) \subseteq \Refl(R_{w})
\]
is a full subcategory. After choosing abelian varieties $A_{v}$ and $A_{w}$ as in 
Proposition~\ref{ConstructionApi} associated to the sets $v$ and $w$ respectively, we obtain a diagram of functors
\begin{equation} \label{eq:changeofWeilstring2}
\xymatrix@M+1ex{
\AV_{w} \ar[rr]^{\Hom_{\bF_p}(-,A_{w})} & & \Refl(R_{w})\\
\AV_{v} \ar@{}[u]|{\displaystyle \bigcup} \ar[rr]^{\Hom_{\bF_p}(-,A_{v})}  && \Refl(R_{v}) \ar@{}[u]|{\displaystyle \bigcup} .
}
\end{equation}
where the vertical functors are natural full subcategories. The diagram \eqref{eq:changeofWeilstring2} need not commute  for arbitrary unrelated choices $A_w$ and $A_v$.
The next proposition shows that for every $A_w$ there is a canonical abelian subvariety $A_{v,w} \subseteq A_w$ that leads to a choice of $A_v$ for which \eqref{eq:changeofWeilstring2}
commutes.

\begin{prop} 
\label{prop:changeofWeilstring} 
Let $w$ be a set of non-real conjugacy classes of Weil $p$--numbers, let
$A_{w}$ be an abelian variety over $\bF_p$ such that $\End_{\bF_p}(A_{w}) = R_{w}$, and let $v\subseteq w$ be any subset. Then the subgroup generated by all images
\[
A_{v,w} : = \left\langle \im(f) \ ; \ f: B \to A_{w}, \ B \in \AV_{v}\right\rangle \subseteq A_{w}
\] 
satisfies the following:
\begin{enumera}
\item \label{propitem:inAVv} $A_{v,w}$ belongs to $\AV_{v}$ and is an abelian subvariety of $A_{w}$.
\item \label{propitem:freeTate} $T_\ell(A_{v,w})$ is free of rank one over $R_{v}\otimes\bZ_\ell$, for all primes $\ell$.
\item \label{propitem:commute} The diagram \eqref{eq:changeofWeilstring2} commutes by choosing $A_{w}$ as the abelian variety associated to $w$ and
$A_v=A_{v,w}$ as that associated to $v$.
\item \label{propitem:onemapenough} $A_{v,w}$ is the image of any map $f: B \to A_w$ such that $w(B) = v$ and $w(\coker(f))=w\setminus v$.
\end{enumera}
\end{prop}

\begin{proof}
Assertion~\ref{propitem:inAVv} is obvious and assertion~\ref{propitem:commute} follows from the natural equality 
\[
\Hom_{\bF_p}(B,A_{v,w}) = \Hom_{\bF_p}(B,A_{w})
\]
for every $B \in \AV_{v}$, since every morphism $f:B\to A_{w}$ takes values in the subvariety $A_{v,w}\subseteq A_{w}$.

Assertion~\ref{propitem:onemapenough} is obvious once we pass to the semisimple category of abelian varieties up to isogeny. Therefore
$f(B)$ and $A_{v,w}$ have the same dimension. Since by definition $f(B) \subseteq A_{v,w}$ we deduce the claimed equality.

It remains to verify assertion~\ref{propitem:freeTate}, which by Proposition~\ref{characterizationApi} is equivalent to $\End_{\bF_p}(A_{v,w})=R_{v}$.
The natural map 
\begin{equation}\label{RpiRsigma}
R_{w} = \End_{\bF_p}(A_{w}) = \rM_{w}(A_{w}) \longrightarrow \rM_{w}(A_{v,w})  = \End_{\bF_p}(A_{v,w})
\end{equation}
factors through the quotient map $\pr_{v,w} : R_{w} \surj R_{v}$. In order to prove \ref{propitem:freeTate}, it is enough to show that \eqref{RpiRsigma} is surjective.
It suffices to verify surjectivity after $-\otimes \bZ_\ell$ for every prime number $\ell$.  

\smallskip 
Assume first that $\ell\neq p$. Let $C$ be the quotient abelian variety  $C = A_{w}/A_{v,w}$.
 There is an exact sequence of reflexive $R_{w}\otimes \bZ_\ell$-modules
\[
0 \to \rT_\ell(A_{v,w}) \to \rT_\ell(A_{w}) \to \rT_\ell(C) \to 0,
\]
and its $\Ext$-sequence contains
\[
\Hom_{R_{w} \otimes \bZ_\ell}(\rT_\ell(A_{w}),\rT_\ell(A_{w})) \to \Hom_{R_{w} \otimes \bZ_\ell}(\rT_\ell(A_{v,w}),
\rT_\ell(A_{w})) \to \Ext^1_{R_{w} \otimes \bZ_\ell}(\rT_\ell(C),\rT_\ell(A_{w})) .
\]
The $\Ext^1$-term vanishes by Lemma~\ref{lem:vanishingext1}. Therefore Theorem~\ref{TateThm} shows the surjectivity of
\[
\rM_w(A_{w}) \otimes \bZ_\ell =  \Hom_{R_{w} \otimes \bZ_\ell}(\rT_\ell(A_{w}),\rT_\ell(A_{w})) \surj \Hom_{R_{w}
\otimes \bZ_\ell}(\rT_\ell(A_{v,w}),\rT_\ell(A_{w})) = \rM_w(A_{v,w}) \otimes \bZ_\ell.
\]

If $\ell=p$, then the inclusion $A_{v,w}\subseteq A_{w}$ gives a surjection of reflexive $R_{w}\otimes\bZ_p$-modules
\[
\rT_p(A_{w}) \surj \rT_p(A_{v,w}).
\]
Since $T_p(A_{w})$ is free over $R_{w}\otimes\bZ_p$, we deduce a surjection
\[
\Hom_{R_{w}\otimes\bZ_p}(T_p(A_{w}), T_p(A_{w}))\surj \Hom_{R_{w}\otimes\bZ_p}(T_p(A_{w}), T_p(A_{v,w})),
\]
which, by Theorem~\ref{TateThm}, says that $\rM_{w}(A_{w})\otimes\bZ_p\to\rM_{w}(A_{v,w})\otimes\bZ_p$ is surjective.
This completes the proof of the proposition. 
\end{proof}

\subsection{Proof of the main result}
\label{sec:proofmain}

We are now ready to prove our main result. We must show that the abelian varieties $A_w$ that exist by 
Proposition~\ref{ConstructionApi} for each $w$, and which yield equivalences of the desired type on the respective full subcategories $\AV_w$ by 
Theorem~\ref{thm:T(A)onAVpi},  can be chosen in a compatible way for every $v \subseteq w$. This requires a two step process. We use the notation of Proposition~\ref{prop:changeofWeilstring}.
\begin{itemize}
\item
First, we establish compatibility on the set theoretic level: we must fix isomorphism classes for each $A_w$, such that $A_v \simeq A_{v,w}$ for every $v \subseteq w$. 
\item
Secondly, we categorize the first choice:  one must choose isomorphisms $A_v \simeq A_{v,w}$ such that the inclusions 
$\ph_{w,v}: A_v \simeq A_{v,w} \subseteq A_w$ obey the cocyle condition $\ph_{w,v} \circ \ph_{v,u} = \ph_{w,u}$ for $u \subseteq v \subseteq w$ and thus construct an ind-system $\dA = (A_w, \ph_{w,v})$.
\end{itemize}

\begin{proof}[Proof of Theorem~\ref{MainThm}]

For any finite set $w\subseteq W_p$ that avoids $\sqrt{p}$,  let $Z(w)$ be the set of isomorphism classes $[A]$ of abelian varieties $A$ in $\AV_{w}$ such that the natural map 
$R_{w} \to \End_{\bF_p}(A)$ is an isomorphism. The elements of $Z(w)$ all belong to the same isogeny class, and so $Z(w)$ is finite, since there are only finitely many isomorphism classes of abelian varieties over a finite field lying in a given isogeny class 
(in fact, finiteness holds for isomorphism classes of abelian varieties of fixed dimension, cf.~\cite{Mi} Corollary~18.9). 
Moreover, the set $Z(w)$ is non-empty by Proposition~\ref{ConstructionApi}. 

For any pair $v\subseteq w$ of finite sets of non-real Weil $p$--numbers, we construct a map
\[
\zeta_{v,w}:Z(w)  \longrightarrow Z(v)
\]
by $\zeta_{v,w}([A]) = [B]$ where $B$ is the abelian subvariety of $A$ generated by the image of all $f: C \to A$ with $w(C) \subseteq v$.  
Proposition~\ref{prop:changeofWeilstring} states that $\zeta_{v,w}$ indeed takes values in $Z(v)$.

These maps satisfy the compatibility condition 
\[
\zeta_{u,w}=\zeta_{u,v} \zeta_{v,w},
\]
for all tuples $u \subseteq v\subseteq w$, hence they define a projective system
\[
(Z(w), \zeta_{v,w})
\]
indexed by finite subsets $w \subseteq W_p$ with $\sqrt{p} \notin w$. Since the sets $Z(w)$ are finite and non-empty, a standard compactness argument shows that the inverse limit is not empty:
\[
Z = \varprojlim_w Z(w) \not= \emptyset.
\]
We choose a compatible\footnote{We will see later in Remark~\ref{rmk:zetasurjective} that $\zeta_{v,w}$ is always surjective. This extra piece of information simplifies the construction of the system marginally. However, we find it conceptually easier to deduce this fact from the anti-equivalence of Theorem~\ref{MainThm}, hence the order of the assertions and proofs.}
system $z = (z_w) \in Z$ of isomorphism classes of abelian varieties. 

\smallskip

Now we would like to choose abelian varieties $A_w$ in each class $z_w$ and inclusions 
\[
\ph_{w,v}: A_v \to A_w
\]
for every $v \subseteq w$ that are isomorphic to the inclusion from Proposition~\ref{prop:changeofWeilstring} in a compatible way: for $u \subseteq v \subseteq w$ we want 
\[
\ph_{w,u}=\ph_{w,v} \ph_{v,u}.
\]

Because the set of Weil numbers is countable, we may choose a cofinal totally ordered subsystem of finite subsets of $W_p^\com$ 
\[
w_1 \subseteq w_2 \subseteq \ldots \subseteq w_i \subseteq \ldots .
\]
Working first with this totally ordered subsystem, we can construct  a direct system 
\[
\mathcal{A}_0=(A_{w_i},\ph_{w_j,w_i})
\]
of abelian varieties as desired by induction. If $A_{w_i}$ is already constructed, then we choose $A_{w_{i+1}}$ in $z_{w_{i+1}}$ and deduce from $\zeta_{w_{i},w_{i+1}}(z_{w_{i+1}}) = z_{w_i}$ that there is an inclusion $\ph_{w_{i+1},w_i}: A_{w_i} \to A_{w_{i+1}}$ as desired.

Once this is achieved, we may identify all transfer maps  of the restricted system $\dA_0$ with inclusions. Then we can define for a general $w$ the abelian variety $A_w$ by $w \subseteq w_i$ for large enough $i$ by means of the construction of 
Proposition~\ref{prop:changeofWeilstring} as the abelian subvariety 
\[
A_w := A_{w,w_i}  \subseteq A_{w_i},
\]
This choice is well defined, i.e., independent of $i \gg 0$.
Furthermore, there are compatible transfer maps $\ph_{v,w} : A_v \to A_w$ for all $v \subseteq w$ that lead to the desired direct system 
\[
\dA = (A_w, \ph_{w,v}).
\]
In the sense of ind-objects we have $\dA_0 \simeq \dA$ and so $\dA_0$ would suffice for Theorem~\ref{MainThm}, but we wanted to restore symmetry and have $A_w$ for all finite subsets $w \subseteq W_p^\com$.

\smallskip

Let $A$ be any element of $\AV_p^{\rm com}$, set
\[
T(A)=\Hom_{\bF_p}(A, \mathcal{A})  = \varinjlim_w \Hom_{\bF_p}(A, A_{w})  = \varinjlim_w \rM_w(A).
\]
The
groups $\Hom_{\bF_p}(A, A_{w})$ are stable when $w$ is large enough. More precisely, if $w$, $w'$
are finite sets of Weil $p$--numbers with $w(A)\subseteq w\subseteq w'$, then the map 
\[
\varphi_{w',w} \circ -  : \Hom_{\bF_p}(A, A_{w})\to\Hom_{\bF_p}(A, A_{w'})
\]
is an isomorphism (cf.~Proposition~\ref{prop:changeofWeilstring}). 
Moreover, $T(-)$ restricted to
$\AV_{w}$ recovers the functor $\rM_{w}(-)$ of Theorem~\ref{thm:T(A)onAVpi} constructed using the object $A_{w}$ of
$\mathcal{A}$, and induces an anti-equivalence between $\AV_{w}$ and $\Refl(R_{w})$.

Observe that, by the naturality of the Frobenius isogeny, 
for any finite $w \subseteq W_p$ avoiding $\sqrt{p}$, 
and for any $f\in\Hom_{\bF_p}(A, A_{w})$ the diagram
\[
\xymatrix@M+1ex{
A \ar[r]^{\pi_A}\ar[d]_f  & A\ar[d]^f\\
A_{w} \ar[r]^{\pi_{A_{w}}}& A_{w}\\
}
\]
is commutative. This implies that, for $w$ sufficiently large, the action of $F_w \in R_{w}$ on $T(A)$
is given by $T(\pi_A)$, the morphism induced by the Frobenius isogeny $\pi_A$ via functoriality of $T$.

Compatibility in $w$ shows that $T(-)$ induces an anti-equivalence
\[
T = \varinjlim M_w :\AV_{p}^{\rm com} = \varinjlim_w \AV_w  \xrightarrow{\sim} \varinjlim_w \Refl(R_{w}) = \Refl( \dR_p^\com).
\]
Due to the remarks of Section~\S\ref{sec:targetcategory}  this is precisely the claim of Theorem~\ref{MainThm} and so its proof is complete.
\end{proof}

\section{Properties of the functor $T$}
\label{sec:properties}

\subsection{Recovering Tate and Dieudonn\'e module}

Let $A$ be an abelian variety over $\bF_p$, and set $w = w(A)$. 
We explain here how the $R_{w}\otimes\bZ_\ell$-modules $T_\ell(A)$ can be recovered from the pair $(T(A), F)$ attached
to $A$ by Theorem~\ref{MainThm}. We set  for all prime numbers $\ell$ 
\[
\mathcal{R}_\ell = \varprojlim_w (R_{w}\otimes\bZ_\ell),
\]
where in the projective limit $w$ ranges through all finite subsets of $W_p^\com$, and define
\[
T_\ell(\mathcal{A}) = \left\{ 
\begin{array}{rl} 
\varinjlim_w T_\ell(A_{w}) & \ell \not= p \\[1ex]
\varprojlim_w T_p(A_w) & \ell = p
\end{array}
\right.
\]
as the direct limit if $\ell\neq p$ and the projective limit if $\ell=p$ of the system obtained by applying $T_\ell(-)$ to the direct system $\mathcal{A} =(A_{w})_w$ constructed in the proof of Theorem~\ref{MainThm}.

We first discuss the $\ell$-adic Tate module and assume $\ell \not= p$. 
Since for $v \subseteq w$ the map $A_v \to A_w$ is an inclusion of abelian varieties, the induced map $T_\ell(A_v) \to T_\ell(A_w)$ is the inclusion of a direct summand, at least as $\bZ_\ell$-modules. Hence $T_\ell(\dA)$ is a free $\bZ_\ell$-module of countable infinite rank.

\begin{prop}  \label{prop:recovertate}
Let $A$ be an abelian variety over $\bF_p$ with $\sqrt{p} \notin w(A)$. There is a natural isomorphism of $\mathcal{R}_\ell$-modules
\[
T_\ell(A)\stackrel{\sim}{\longrightarrow}\Hom_{\mathcal{R}_\ell}(T(A)\otimes\bZ_\ell, T_\ell(\mathcal{A})).
\]
\end{prop}

\begin{proof}
Let $w \subseteq W_p^\com$ be a finite set containing $w(A)$. Since  $R_w \otimes \bZ_\ell$ is Gorenstein, dualizing 
\eqref{eq:elladictatefromT}  yields the first equality in
\[
T_\ell(A) = \Hom_{R_w \otimes \bZ_\ell}(M_{w}(A) \otimes \bZ_\ell,T_\ell(A_w)) = \Hom_{\dR_\ell}(T(A) \otimes \bZ_\ell, T_\ell(\dA)).
\]
The second equality holds, because $T_\ell(A_w) \subseteq T_\ell(\dA)$ is the maximal submodule on which $\dR_\ell$ acts through its quotient $\dR_\ell\to R_w \otimes \bZ_\ell$. 
\end{proof}

Now we address the contravariant Dieudonn\'e module $T_p(A)$. We endow $T_p(\dA)$ with the projective limit topology. If $M$ is a topological $\dR_p$--module
which is finite and free over $\bZ_p$ then the action of $\dR_p$ on $M$ factors through $\dR_p\to R_w \otimes \bZ_p$ for some large enough $w$, by compactness
of $M$. We denote by
\[
M\hat{\otimes}_{\dR_p} T_p(\dA) = \varprojlim_{w \gg \emptyset} M \otimes_{R_w \otimes \bZ_p} T_p(A_w)
\]
the continuous tensor product.

\begin{prop}  \label{prop:recoverdieudonne}
Let $A$ be an abelian variety over $\bF_p$ with $\sqrt{p} \notin w(A)$.  There is a natural isomorphism of $\mathcal{R}_p$-modules
\[
T_p(A) = (T(A) \otimes \bZ_p) \hat{\otimes}_{\dR_p} T_p(\dA).
\]
\end{prop}
\begin{proof} 
Let $w \subseteq W_p^\com$ be a finite set containing $w(A)$. 
We deduce from \eqref{eq:dieudonnefromT} a natural identification
\begin{align*}
T_p(A) & = \Hom_{R_w \otimes \bZ_p}(T_p(A_w),T_p(A)) \otimes_{R_w \otimes \bZ_p} T_p(A_w) \\
& = \rM_w(A) \otimes_{R_w} T_p(A_w) = (T(A) \otimes \bZ_p) \hat{\otimes}_{\dR_p} T_p(\dA),
\end{align*}
because for $w(A) \subseteq w \subseteq w'$ the natural maps 
\[
(T(A) \otimes \bZ_p) \otimes_{R_{w'} \otimes \bZ_p} T_p(A_{w'}) \to (T(A) \otimes \bZ_p) \otimes_{R_{w} \otimes \bZ_p} T_p(A_{w})
\]
are isomorphisms.
\end{proof}

\subsection{Isogenies and inclusions}
\label{sec:isogenyinclusion}

We discuss how the functor $T(-)$ detects isogenies and inclusions.

\begin{prop} \label{prop:Tonisogenies}
Let $A$ and $B$ be abelian varieties in $\AV_p^{\rm com}$. 
\begin{enumerate}
\item
The map $f:B \to A$ is an isogeny, if and only if  $T(f) \otimes \bQ$ is an isomorphism.
\item 
For an isogeny $f: B \to A$ the map $T(f)$ is injective and the image is of index
\[
\deg(f) = |\coker(T(f))|.
\]
\end{enumerate}
\end{prop}
\begin{proof}
(1) 
An isogeny $f$ has an inverse up to a multiplication by $n$ map for $n=\deg(f)$. Therefore $T(f)$ is an isomorphism after inverting $\deg(f)$. 

Conversely, if $f$ is not an isogeny, then either $\ker(f)$ or $\coker(f)$ have a non-trivial abelian variety as a direct summand up to isogeny. In the presence of such a direct summand the map $T(f) \otimes \bQ$ cannot be an isomorphism.

(2) 
We indicate the $\ell$-primary part by an index $\ell$. Then using Proposition~\ref{prop:recovertate}, for $\ell \not= p$, we have
\[
|\coker(T(f))|_\ell = |\coker(T(f)) \otimes \bZ_\ell| = |\coker (T_\ell(f)^\vee : T_\ell(A)^\vee \to T_\ell(B)^\vee)|.
\]
The duals here are $\Hom(-,R_w\otimes \bZ_\ell)$. Since $R_w \otimes \bZ_\ell$ is reduced Gorenstein of dimension $1$,  we can use \cite{Ba} Theorem 6.3 (4) and induction on the length  to see that 
\[
|\coker (T_\ell(f)^\vee : T_\ell(A)^\vee \to T_\ell(B)^\vee)| = |\coker (T_\ell(f) : T_\ell(B) \to T_\ell(A))| = |\ker(f)|_\ell.
\]
If $\ell = p$, using Proposition~\ref{prop:recoverdieudonne} yields
\[
|\coker(T(f))|_p = |\coker(T(f)) \otimes \bZ_p| = |\coker (T_p(f) : T_p(A) \to T_p(B))| = |\ker(f)|_p,
\]
where the last equality follows from Dieudonn\'e theory.
\end{proof}

\begin{prop} 
\label{prop:Tonsurjectivemaps}
Let $A$ and $B$ be abelian varieties in $\AV_p^{\rm com}$. For a map $f: B \to A$ the following are equivalent.
\begin{enumerate}
\item[(a)] $T(f) : T(A) \surj T(B)$ is surjective.
\item[(b)] The map $f$ can be identified with the inclusion of an abelian subvariety.
\end{enumerate}
\end{prop}
\begin{proof}
If $T(f)$ is surjective, Proposition~\ref{prop:recovertate} shows that the induced map $T_\ell(B) \to T_\ell(A)$ is injective. 
Therefore $\ker(f)$ is at most a finite group scheme. We may therefore replace $A$ by the image $A_0$ of $B \to A$ and thus reduce to the case of the isogeny $f_0: B \to A_0$. Here Proposition~\ref{prop:Tonisogenies} implies that $\deg(f_0) = 1$, hence $B = A_0$ and $f$ is indeed an inclusion of an abelian subvariety.

Conversely, if $f: B \to A$ is an inclusion, then there is a map $g:A \to B$ such that $gf : B \to B$ is an isogeny. Therefore $T(f)$ has at least an image of finite index. The image of $T(f)$ is a reflexive submodule in the image of the equivalence $T(-)$, so that there is an abelian variety $C$ and a factorization $B \to C \to A$ with $T(A) \surj T(C)$ surjective and $T(C) \subseteq T(B)$ an inclusion. 

We have already proven that $C \to A$ is an abelian subvariety, and it is easy to see that $B \to C$ is an isogeny. Therefore $B \to C$ is an isomorphism and the proof is complete.
\end{proof}

As an application we prove a variant for objects of $\AV_p$ of Waterhouse's theorem on possible endomorphism rings of $\bF_p$-simple abelian varieties over $\bF_p$, see \cite{Wa} Theorem~6.1.2.

\begin{thm}
Let $w$ be a set of conjugacy classes of non-real Weil $p$--numbers. Then the following are equivalent.
\begin{enumeral}
\item $S$ is an order in $R_w \otimes \bQ$ containing $R_w$.
\item $S$ is isomorphic as an $R_w$-algebra to $\End_{\bF_p}(B)$ for an abelian variety $B$ with $w(B)= w$ and such that all its simple factors up to isogeny occur with multiplicity $1$.
\end{enumeral}
\end{thm}
\begin{proof}
Since $R_w$ is the minimal central order for abelian varieties $B$ with $w(B) = w$, it is clear that (b) implies (a). 

Conversely, if $S$ is an order containing $R_w$, then $S$ is a reflexive $R_w$-module and thus corresponds to an abelian variety $B$. Let $A_w$ be the abelian variety occuring in the ind-system pro-representing $T(-)$, so that $T(A_w) = R_w$. The inclusion $R_w \subseteq S$ corresponds to an isogeny $\ph : B \to A_w$ by Proposition~\ref{prop:Tonisogenies}, so that $B$ has 
the required Weil support and product structure up to isogeny. Moreover, 
\[
\End_{\bF_p}(B) = \End_{R_w}(S) = \{ \lambda \in R_w \otimes \bQ \ ; \ \lambda S \subseteq S\} = S
\]
shows (a) implies (b).
\end{proof}

\section{Ambiguity and comparison}
\label{sec:ambiguity}

The construction of the functor $T(-)$ in Section~\S\ref{sec:proofmain} depends on the choice of an ind-abelian variety $\dA$. For the sake of distinguishing the different choices we denote in this section
\[
T_\dA(-) = \Hom_{\bF_p}(-,\dA).
\]

\subsection{Continuous line bundles}
\label{sec:pic}

\begin{defi}
Let $W \subseteq W_q$ be a subset.  Let us denote by $\dR_W$ the pro-ring $(R_w)$ where $w$ ranges over the finite subsets of $W$.
\begin{enumerate}
\item
An \textbf{$\dR_W$-module} is a pro-system $\dM = (M_w)$ with $w$ ranging over the finite subsets of $W$, such that $M_w$ is an $R_w$-module and the maps $M_w \to M_v$ for $v \subseteq w$ are $R_w$-module homomorphisms (where $R_w$ acts on $M_v$ via the projection $R_w \to R_v$). Homomorphisms of $\dM$ are levelwise $R_w$-module homomorphisms.

\item
An $\dR_W$-module $\dM$ is \textbf{invertible} if for all $w \subseteq W$ the $R_w$-module $M_w$ is invertible
and for $v \subseteq w$ the maps $M_w \to M_v$ are surjective (equivalently: induce a natural isomorphism $M_w \otimes_{R_w} R_v \simeq M_v$).

\item
The set of isomorphism classes of invertible $\dR_W$-modules forms a group under levelwise tensor products that we denote by $\Pic(\dR_W)$, the \textbf{Picard group} of $\dR_W$.
\end{enumerate}
\end{defi}

For a finite set $w$ of conjugacy classes of Weil $q$-numbers, we set $X_w = \Spec(R_w)$ and consider the ind-schemes
\[
\cX = \varinjlim_w X_w
\]
and for a subset $W \subseteq W_q$ the ind-scheme
\[
\cX_W = \varinjlim_{w \subseteq W} X_w.
\]
with closed immersions as transfer maps, all denoted $i$, induced by the projections $\pr_{v,w} : R_w \surj R_v$. The invertible $\dR_W$-modules are nothing but line bundles on $\cX_W$ and
\[
\Pic(\dR_W) = \Pic(\cX_W) = \rH^1(\cX_W,\dO^\times).
\]
Since $\dO_{\cX_W}^\times = \varprojlim_{w \subseteq W}  i_\ast \dO_{X_w}^\times$ we find an exact sequence
\[
0 \to \varprojlim_{w \subseteq W} \! \! \! ^1 \ R^\times_{w} \to \Pic(\dR_W) \to \varprojlim_{w \subseteq W} \Pic(R_{w}) \to 0. 
\]
The quotient of $\Pic(\dR_W)$ given by $\varprojlim_w \Pic(R_{w})$ parametrizes 
the choices of a compatible system of isomorphism classes of rank one $R_w$-modules $M_w$.  The $\varprojlim^1$-term parametrizes all choices of transfer maps to obtain
an invertible $\dR_W$-module $\mathcal{M}=(M_w)$ from a given compatible choice of isomorphism classes of invertible $R_w$-modules at every level.

\begin{prop}
\label{prop:picrestrictionsurjective}
Let $V \subseteq W \subseteq W_q$ be subsets. Then the natural restriction map 
\[
\Pic(\dR_W) \surj \Pic(\dR_V)
\]
is surjective.
\end{prop}
\begin{proof}
For $v \subseteq w$ define Zariski sheaves $ \dK_{v,w}$ on $\cX_W$ by the  short exact sequence
\[
0 \to \dK_{v,w} \to  i_\ast \dO_{X_w}^\times  \to  i_\ast \dO_{X_v}^\times \to 0
\]
Then $\dK_{V,W} = \varprojlim_{w \subseteq W} \dK_{w \cap V,w}$ is the kernel of 
$\dO^\times_{\cX_W} \surj i_\ast \dO_{\cX_V}^\times$.
The Zariski cohomology sequence yields an exact sequence
\[
\Pic(\dR_W) \to \Pic(\dR_V) \to \rH^2(\cX_W,\dK_{V,W})
\]
and it remains to show vanishing of $\rH^2(\cX_W,\dK_{V,W})$. The pro-structure of $\dK_{V,W}$ leads to a short exact sequence
\[
0 \to \varprojlim_{w \subseteq W} \! \! \! ^1 \rH^1(X_w, \dK_{w \cap V,w}) \to \rH^2(\cX_W,\cK_{V,W}) \to \varprojlim_{w \subseteq W} \rH^2(X_w, \dK_{w \cap V,w}) \to 0.
\]

The $\varprojlim$-term on the right vanishes by cohomological dimension because $\dim(X_w) = 1$.  The $\varprojlim^1$-term on the left vanishes, because we claim that $(\rH^1(X_w, \dK_{w \cap V,w}))_{w \subseteq W}$ is a surjective system, and hence a Mittag--Leffler system. Indeed, for finite subsets $w \subseteq w' \subseteq W$, the cokernel $\cC_{w,w'}$ of 
\[
\dK_{w' \cap V,w'} \to \dK_{w \cap V,w}
\]
is a sheaf with support in at most the finitely many points of $X_{w'}$ that are contained in more than one irreducible component and so $\rH^1(X_{w'},\cC_{w,w'}) = 0$ . Since $\rH^1(X_{w'},-)$ is right exact, we have an exact sequence
\[
\rH^1(X_{w'},\dK_{w' \cap V,w'}) \to \rH^1(X_w,\dK_{w \cap V,w}) \to \rH^1(X_{w'},\cC_{w,w'}) = 0
\]
from which we deduce the claim.
\end{proof}

\subsection{Mixed tensor products}
\label{sec:tensor}

We recall Serre and Tate's well known tensor product construction (see \cite{giraud:hom} for the parallel $\Hom$-construction explaining a construction of Shimura and Taniyama). Let $A$ be an abelian variety over $\bF_q$ and $M$ be a finitely generated $R_w$-module for some $w(A) \subseteq w \subset W_q$. The $R_w$-action on $A$ induces an $R_w$-module structure on the set of $U$-valued points for any $\bF_q$-scheme $U$. 
The fppf-sheafification $\widetilde{M \otimes_{R_w} A}$ of the functor on $\bF_q$-schemes
\[
U \mapsto M \otimes_{R_w} A(U)
\]
is representable by an abelian variety. Indeed, let 
\[
R_w^m \xrightarrow{\ph} R_w^n \to M \to 0
\]
be a finite presentation. The $m \times n$-matrix $\ph$ also defines a map $\ph_A: A^m \to A^n$, and 
\[
M \otimes_{R_w} A(U) = \coker(\ph \otimes \id_{ A(U)}) = \coker(\ph_A : A(U)^m \to A(U)^n) = \coker(\ph_A(U))
\]
so that 
\[
\widetilde{M \otimes_{R_w} A} = \coker(\ph_A)
\]
and this is representable by an abelian variety. We denote the representing object by 
\[
M \otimes_{R_w} A.
\]

If $w \subseteq w'$ and $M'$ is a finitely presented $R_{w'}$-module with $M = M' \otimes_{R_{w'}} R_w$, then  there is an obvious identification
\[
M' \otimes_{R_{w'}} A = M \otimes_{R_w} A.
\]
In particular, if $W \subseteq W_q$ is a subset and $w(A) \subseteq W$, then for any invertible $\dR_W$-module $\dM=(M_w)$ we have a well defined tensor product by 
\[
\dM \otimes_{\dR_W} A := M_w \otimes_{R_w} A
\]
for all sufficiently large finite $w(A) \subseteq w \subseteq W$.

\subsection{Choices of ind-representing objects}
\label{sec:choicesofindrepresentingobjects}

Before we  describe our choices, we need three propositions of independent interest.

\begin{prop}
\label{prop:Homandtensor}
Let $W \subseteq W_q$ be a subset,  $A$ be an abelian variety with $w(A) \subseteq W$, and let $\dM = (M_w)$ be an invertible $\dR_W$-module. Then there is a natural isomorphism
\[
\Hom_{\bF_q}(-, \dM \otimes_{\dR_W} A) \simeq  \dM \otimes_{\dR_W} \Hom_{\bF_q}(-,A) 
\]
 of functors $\AV_W \to \Refl(\dR_W)$.
\end{prop}
\begin{proof}
We set $w = w(A)$ and must show for any abelian variety $X$ over $\bF_q$ naturally
\[
\Hom_{\bF_q}(X,M_w \otimes_{R_w} A) = M_w \otimes_{R_w} \Hom_{\bF_q}(X,A).
\]
We extend this claim to projective $R_w$-modules $M$ of finite rank. Since the tensor construction is compatible with direct sums, clearly the claim is additive in $M$ in the sense that it holds for $M'$ and $M''$ if and only if it holds for $M = M' \oplus M''$. This reduces the claim to free modules $M=R_w^n$ and with the same argument to $M= R_w$. Now the claim trivially holds.
\end{proof}

\begin{prop}
\label{prop:anyequivalencerepresentable}
Let $W \subseteq W_q$ be a subset containing no rational Weil $q$--number. Any $\dR_W$-linear contravariant equivalence 
\[
S : \AV_W \to \Refl(\dR_W)
\]
is ind-representable, i.e., of the form
\[
S(-) = \Hom_{\bF_p}(-,\dB)
\]
for an ind-system $\dB = (B_w, \ph_{w,v})$ such that for all finite subsets $v \subseteq w \subseteq W$  the following holds.
\begin{enumer}
\item $w(B_w) = w$.
\item The natural map $R_w \to \End_{\bF_q}(B_w)$ is an isomorphism.
\item \label{propitem:multiplicity1} $B_w$ is isogenous to the product of its simple factors with multiplicity $1$.
\item \label{propitem:transfersareinclusionabstractcase} The maps $\ph_{w,v} : B_v \to B_w$ are inclusions.
\end{enumer}
\end{prop}
\begin{proof}
The pro-system $\dR_W = (R_w,\pr_{v,w})$ can be considered as the pro-system of the free rank $1$ modules $R_w \in \Refl(R_w) \subseteq \Refl(\dR_W)$. As such there is a unique ind-system $\dB = (B_w, \ph_{w,v})$ with $S(\dB) = (S(B_w)) = \dR_W$. Yoneda's lemma assigns to the compatible elements $1 \in R_w = S(B_w)$ a natural transformation
\[
\Phi:  \Hom_{\bF_q}(-,\dB) = \varinjlim_{w} \Hom_{\bF_q}(-,B_w) \longrightarrow S(-).
\]
For every $A \in \AV_W$ the map $\Phi$ is the composition of the two isomorphims
\[
\varinjlim_w \Hom(A,B_w) \xrightarrow{S} \varinjlim_w \Hom_{R_w}(R_w,S(A) ) \xrightarrow{\ev_1} S(A)
\]
where $\ev_1$ denotes the evaluation map at $1$. It remains to prove the finer claims on the ind-representing system $\dB$.

\smallskip

Since $S$ is an $\dR_W$-linear equivalence, $\dR_W$ acts on $B_w$ through $R_w$ as on $S(B_w) = R_w$. Here we use that $R_w$ is commutative and so we can forget to pass to the opposite ring due to $S$ being contravariant. Since $F_w$ acts on $B_w$ by the Frobenius isogeny $\pi_{B_w}$, and on $R_w = S(B_w)$ by $F_w \in R_w$ it follows that $w(B_w) = w$. 

\smallskip

The natural map $R_w \to \End_{\bF_w}(B_w)$ is an isomorphism,  because applying the $\dR_W$-linear $S(-)$ transforms it to the map $R_w \to \End_{\dR_w}(R_w)$ which is an isomorphism indeed. We deduce from this also assertion~\ref{propitem:multiplicity1}.

\smallskip

It remains to show  that $\ph_{w,v} : B_v \to B_w$ is isomorphic to an inclusion for all $v \subseteq w$. We denote the image of $\ph_{w,v}$ by $C$. Since $S$ is ind-representable, the surjection $B_v \to C$ becomes an inclusion 
\[
S(C) \inj S(B_v).
\]
Since by construction $S(B_w) \to S(B_v)$ is the surjective map $\pr_{v,w}: R_w \to R_v$ we conclude the in fact $S(C) \simeq S(B_v)$  is an isomorphism. Consequently, because $S$ is an equivalence, we have  $C \simeq B_v$ and 
assertion~\ref{propitem:transfersareinclusionabstractcase}  holds.
\end{proof}

The third proposition is related to Proposition~\ref{characterizationApi}.

\begin{prop} 
\label{prop:Toninvertibleobjects}
Let $W \subseteq W_q$ be a subset containing no rational Weil $q$--number, and let 
\[
S : \AV_W \to \Refl(\dR_W)
\]
be an $\dR_W$-linear contravariant equivalence.

Let $w \subseteq W$ be a finite set of conjugacy classes of Weil $q$--numbers, and let $A$ be an abelian variety over $\bF_q$ with $w = w(A)$. The following are equivalent.
\begin{enumerate}
\item[(a)] The natural map $R_w \to \End_{\bF_q}(A)$ is an isomorphism.
\item[(b)] $S(A)$ is a projective $R_w$-module of rank $1$.
\end{enumerate}
\end{prop}
\begin{proof}
Since $S(-)$ is an equivalence of categories, the map $R_w \to \End_{\bF_p}(A)$ is an isomorphism if and only if the map 
\[
R_w \to \End_{R_w}(S(A)) 
\]
is an isomorphism ($S$ is contravariant but the rings are commutative here). Since $R_w$ is a reduced Gorenstein ring of dimension $1$ by Theorem~\ref{thm:gorenstein}, this is equivalent by  Proposition~\ref{prop:EndCriterionForInvertible} to  $S(A)$ being a projective $R_w$-module of rank $1$.
\end{proof}

We define the tensor product of an invertible $\dR_W$-module $\dM = (M_w)$ and an ind-system $\dA = (A_w,\ph_{w,v})$ of abelian varieties  indexed by finite subsets of $W$ and with $w(A_w) = w$ by 
\[
\dM \otimes \dA  := (M_w \otimes_{R_w} A_w).
\]

\begin{thm}
\label{thm:ambiguity}
Let $W \subseteq W_q$ be a subset containing no rational Weil $q$--number. 

Let $\dA = (A_w, \ph_{w,v})$ be an ind-system of abelian varieties over $\bF_q$  indexed by finite subsets of $W$ such that 
\begin{enumer}
\item $w(A_w) = w$.
\item The natural map $R_w \to \End_{\bF_q}(A_w)$ is an isomorphism.
\item $A_w$ is isogenous to the product of its simple factors with multiplicity $1$.
\item The maps $\ph_{w,v} : A_v \to A_w$ are inclusions.
\end{enumer}

For an invertible $\dR_W$-module $\dM$, the ind-system $\dM \otimes_{\dR_W} \dA$ has the same properties (i)--(iv), and the group $\Pic(\dR_W)$ acts freely and transitively by 
\[
\dA \mapsto \dM \otimes_{\dR_W} \dA
\]
on the set of isomorphism classes of such ind-systems.
\end{thm}

\begin{rmk}
If $q = p$ and $W = \{\pi\}$ consists of a single  Weil $p$-number, then in this case Theorem~\ref{thm:ambiguity} is a special case of \cite{Wa}, Theorem~6.1.3, from which
the above result is inspired.
\end{rmk}

\begin{proof}[Proof of Theorem~\ref{thm:ambiguity}]
By a $W$-version of the proof of Theorem~\ref{MainThm} for any ind-system $\dA$ satisfying (i)--(iv) the functor 
\[
T_\dA = \Hom_{\bF_q}(-,\dA) : \AV_W \to \Refl(\dR_W)
\]
is a contravariant $\dR_W$-linear anti-equivalence $\AV_W \to \Refl(\dR_W)$. The effect of the action by $\dM \in \Pic(\dR_W)$ on the represented functors is described by Proposition~\ref{prop:Homandtensor} as
\[
T_{\dM \otimes_{\dR_W} \dA}(-)  = \dM \otimes_{\dR_W} T_\dA(-).
\]
Since $\dM=(M_w)$ is invertible, the functor  $\dM \otimes_{\dR_W}-$ is an auto-equivalence of $\AV_W$. We thus have natural isomorphisms
\[
R_w = \End_{\bF_q}(A_w) = \End_{\bF_q}(\dM \otimes_{\dR_W} A_w) = T_{\dM \otimes_{\dR_W} \dA}(\dM \otimes_{\dR_W} A_w).
\] 
Moreover, since $\dM=(M_w)$ is invertible, the functor $T_{\dM \otimes_{\dR_W} \dA}(-)$ is an anti-equivalence as well, and 
\[
T_{\dM \otimes_{\dR_W} \dA}(\dM \otimes_{\dR_W} \dA) = \dR_W
\]
as pro-systems. It follows from the proof of Proposition~\ref{prop:anyequivalencerepresentable} that $\dM \otimes_{\dR_W} \dA$ also satisfies properties (i)--(iv). This shows that $\Pic(\dR_W)$ indeed acts on isomorphism classes of such $\dA$.

\smallskip

Let $\dM$ be an invertible $\dR_W$-module and let $\dA$ be a pro-system as above such that there is an isomorphism  $\dM \otimes_{\dR_W} \dA \simeq \dA$. Evaluating the resulting natural isomorphism 
\[
\dM \otimes_{\dR_W} T_\dA(-) \simeq T_{\dA}(-).
\]
in $\dA$ itself yields an isomorphism $\dM \otimes_{\dR_W} \dR_W \simeq \dR_W$ and hence $\dM$ must be trivial in 
$\Pic(\dR_W)$. This shows that the action is free.

\smallskip

Let now $\dA$ and $\dB$ be two pro-systems of the type considered. The $\dR_W$-module
\[
\dM = T_{\dB}(\dA) = (\Hom_{\bF_p}(A_w,B_w))
\]
(note that all maps of pro-objects $\dA \to \dB$ are levelwise maps due to $w(A_w) = w = w(B_w)$) is levelwise an invertible $R_w$-module $\dM_w = T_\dB(A_w)$ by Proposition~\ref{prop:Toninvertibleobjects}. The transfer maps $\dM_w \to \dM_v$ agree with $T_{\dB}(\ph_{w,v})$ which is surjective. Indeed, the image corresponds to an abelian variety $C$ such that $\ph_{w,v}$ factors as
\[
A_v \to C \to A_w.
\]
Now the same argument as in the proof of Proposition~\ref{prop:anyequivalencerepresentable} shows that 
$w(C) \subseteq w$ and $C \to A_w$ an inclusion. Since $\ph_{w,v}$ is an inclusion we necessarily have $A_v = C$ and $T_{\dB}(\ph_{w,v})$ is indeed surjective. Consequently, the $\dR_W$-module $\dM = (\dM_w)$ is invertible.

\smallskip

There is a natural map defined by composition of maps
\[
\dM \otimes_{\dR_W} T_\dA(-) = \Hom(\dA,\dB) \otimes \Hom(-,\dA) \xrightarrow{} \Hom(-,\dB) = T_\dB(-).
\]
This is an isomorphism, because for every $X$ in $\AV_W$ we have with $w$ large enough
\begin{align*}
\dM \otimes_{\dR_W} T_\dA(X) & = \Hom_{\bF_p}(A_w,B_w) \otimes_{R_w} \Hom_{\bF_p}(X,A_w)  \\
& = T_\dB(A_w) \otimes_{R_w} \Hom_{R_w}(T_\dB(A_w),T_\dB(X)) \\
& = T_\dB(X).
\end{align*}
Here we have used again the assumption that $T_\dB(-)$ is an equivalence and the fact that $T_\dB(A_w)$ is invertible as an $R_w$-module by Proposition~\ref{prop:Toninvertibleobjects}.
\end{proof}

\begin{cor}
There is a free and transitive action of $\Pic(\dR_p^\com)$ on the isomorphism classes of  ind-systems $\dA$ that represent $\dR_p^\com$-linear anti-equivalences $\dA_p^\com \to \Refl(\dR_p^\com)$.
\end{cor}
\begin{proof}
This is immediate from Theorem~\ref{thm:ambiguity}, the proof of Theorem~\ref{MainThm} and 
Proposition~\ref{prop:anyequivalencerepresentable}.
\end{proof}

\begin{rmk}
\label{rmk:zetasurjective}
With the notation of Section~\S\ref{sec:proofmain}, for finite sets $v \subseteq w \subseteq W_p$ avoiding $\sqrt{p}$ the transfer map 
\[
\zeta_{v,w} : Z(w) \to Z(v)
\]
in the pro-system of isomorphism classes occuring in the proof of Theorem~\ref{MainThm} is in fact surjective. This follows immediately from Theorem~\ref{thm:ambiguity} and the surjectivity of $\Pic(R_w) \to \Pic(R_v)$ from 
Proposition~\ref{prop:picrestrictionsurjective}. 
\end{rmk}

\begin{cor}
\label{cor:extendsystems}
Let $V \subseteq W \subseteq W_p$ be subsets avoiding $\sqrt{p}$ and let $\dA_V = (A_v,\ph_{w,v})$ be an ind-system of abelian varieties over $\bF_p$ indexed by finite subsets of $V$ as in Theorem~\ref{thm:ambiguity} such that 
\[
T_{\dA_V}  = \Hom_{\bF_p}(-,\dA_V) : \AV_V \to \Refl(\dR_V)
\]
is an $\dR_V$-linear anti-equivalence of categories. Then $\dA_V$ can be extended to an ind-system $\dA_W = (A_w,\ph_{v,w})$ of abelian varieties over $\bF_p$ indexed by finite subsets of $W$ as in Theorem~\ref{thm:ambiguity}. In particular the anti-equivalence 
\[
T_{\dA_W} = \Hom_{\bF_p}(-,\dA_W) : \AV_W \to \Refl(\dR_W)
\]
naturally extends $T_{\dA_V}$.
\end{cor}
\begin{proof}
We start by choosing an auxiliary ind-system $\dB_W$ indexed by finite subsets of $W$ as in Theorem~\ref{thm:ambiguity}. The restriction 
\[
\Hom_{\bF_p}(- , \dB_W) : \AV_V \to \Refl(\dR_V)
\]
is an $\dR_V$-linear anti-equivalence and ind-represented by the restriction $\dB_V = \dB_W|_V$ of the indices to finite subsets of $V$. By Theorem~\ref{thm:ambiguity} there is an $\dM_V \in \Pic(\dR_V)$ such that 
\[
\dA_V = \dM_V \otimes_{\dR_V} \dB_V.
\]
By Proposition~\ref{prop:picrestrictionsurjective} we can find $\dM_W \in \Pic(\dR_W)$ such that $\dM_V = \dM_W \otimes_{\dR_W} \dR_V$. Then 
\[
\dA_W = \dM_W \otimes_{\dR_w} \dB_W
\]
obviously extends $\dA_V$ in the desired manner.
\end{proof}

\subsection{Comparison with Deligne's functor for ordinary abelian varieties over $\bF_p$}
\label{sec:comparedeligne}
Let  $w\subseteq W_p^\com$ be a finite subset, and let $\tau:R_w\to R_w$ be the automorphism interchanging $F_w$ and $V_w$.
Denote by $R_w^\tau$ the $R_w$--module obtained by letting $R_w$ operate onto itself via $\tau$. Similarly, for an object
$M$ of $\Refl(R_w)$ denote by $M^\tau$ the $R_w$--module $M\otimes_{R_w}R_w^\tau$.

We fix a contravariant equivalence $T$ as in Theorem~\ref{MainThm} and an ind-representing system $\dA = (A_w, \ph_{w',w})$ for $T=T_\dA$. 
The covariant functor on $\AV_p^{\rm com}$
\[
T_\ast(A) = T(A^t)^\tau = \varinjlim_w \Hom(A_w^t, A),
\]
is pro-representable by the dual system $\dA^t = (A_w^t, \ph_{w',w}^t)$ and a version of Theorem~\ref{MainThm} with a covariant equivalence 
\[
T_\ast : \AV_p^\com \to \Refl(\dR_p^\com)
\]
holds. Notice that $T_\ast$ is $\dR_p^\com$-linear, since the dual of the Frobenius isogeny $\pi_A:A\to A$ is the Verschiebung isogeny
$p/\pi_{A^t}:A^t\to A^t$.

We recall that Deligne's functor $T_{\rm Del}$ on $\AV_q^{\ord}$ is defined as
\[
T_{\rm Del}(A) = \rH_1(\tilde{A}(\bC),\bZ),
\]
where $\tilde{A}/W(\ov{\bF}_p)$ is the Serre--Tate canonical lift of $A \otimes_{\bF_q} \ov{\bF}_p$ to characteristic $0$ over the Witt-vectors $W(\ov{\bF}_p)$, and where the $\bC$-valued points are taken with respect to an a priori  fixed embedding $W(\ov{\bF}_p) \inj \bC$. The lattice $T_{\rm Del}(A)$ comes equipped with a natural Frobenius action by 
$F = T_{\rm Del}(\pi_A)$.  

\smallskip

Note that the functor depends on the chosen embedding $W(\ov{\bF}_p) \inj \bC$.

\smallskip

We denote by $W_q^{\ord}$ the set of conjugacy classes of ordinary Weil $q$-numbers, i.e., of Weil $q$-numbers such that at least half of the roots of the characteristic
polynomial are $p$--adic units, when regarded inside an algebraic closure of $\bQ_p$. With the abbreviation $\dR_q^{\ord} = \dR_{W_q^{\ord}}$ the main result of \cite{De}
can be stated as follows.

\begin{thm}[\cite{De}~\S7]
\label{DeligneOrd} 
The covariant functor $T_{\rm Del}$ induces an $\dR_q^{\ord}$-linear equivalence of categories 
\[
T_{\rm Del} : \AV_q^{\ord} \to \Refl(\dR_q^{\ord}).
\]
\end{thm}

We now compare $T_\ast(-)$ with   $T_{\rm Del}$ when both are restricted to $\AV_p^{\ord}$.

\begin{prop}
The functor $T_{\rm Del}(-)$ is pro-representable by a pro-system $\dA_{\rm Del}$ and 
\[
T_{\rm Del}(\dA_{\rm Del}) = \dR_q^{\ord}.
\]
The dual ind-system $\dA_{\rm Del}^t$ satisfies (i)--(iv) of Proposition~\ref{prop:anyequivalencerepresentable}.
\end{prop}
\begin{proof}
This follows from Proposition~\ref{prop:anyequivalencerepresentable} applied to the functor $X \mapsto T_{\rm Del}(X^t)$.
\end{proof}

Let $T_{\ast}^{\ord}$ (resp.\ $T_{{\rm Del},p}$) denote the restriction of $T_\ast$ (resp.\ $T_{\rm Del}$) to $\AV_p^{\ord}$. The functor $T_\ast^{\ord}$ is
pro-represented by the dual $\dA^{\ord,t}$ of the ind-system $\dA^{\ord}$ which is defined as $\dA$ restricted to indices in $W_p^{\ord}$.

\begin{prop}
\label{prop:compareDeligne}
There is an invertible $\dR_p^{\ord}$-module $\dM = (M_w)_{w \in W_p^{\ord}}$ and a natural isomorphism
\[
\dM \otimes_{\dR_p^{\ord}} T_{{\rm Del},p}(-)     \xrightarrow{\sim} T_\ast^{\ord}(-)
\]
of covariant equivalences $\AV_p^{\ord} \to \Refl(\dR_p^{\ord})$, and a natural isomorphism of ind-systems
\[
\dM \otimes_{\dR_q^{\ord}} \dA_{\rm Del}^t \simeq \dA^{\ord}.
\]
\end{prop}
\begin{proof}
This follows from Theorem~\ref{thm:ambiguity} applied to $W = W_q^{\ord}$. 
\end{proof}

\begin{prop}
\label{prop:extendDeligne}
For an appropriate choice of ind-system $\dA = (A_w,\ph_{v,w})$, the covariant functor $T_\ast$ associated to the functor $T = T_\dA$ of Theorem~\ref{MainThm} extends a given choice of Deligne's functor 
\[
T_{{\rm Del},p} \simeq T_\ast|_{\dA_p^{\ord}} : \AV_p^{\ord} \to \Refl(\dR_p^{\ord}).
\]
\end{prop}
\begin{proof}
This follows from Proposition~\ref{prop:compareDeligne} together with the argument of Corollary~\ref{cor:extendsystems} based on the surjectivity $\Pic(\dR_p^\com) \to \Pic(\dR_p^{\ord})$ of Proposition~\ref{prop:picrestrictionsurjective}.
\end{proof}


\end{document}